\documentclass[12pt, reqno]{amsart}
\usepackage[mathscr]{eucal}
\usepackage{amscd}
\usepackage{amsfonts}
\usepackage{amsmath, amsthm, amssymb}
\usepackage{latexsym}
\usepackage[dvips]{graphics}
\usepackage{amstext}
\usepackage[all]{xy}
\usepackage[dvipdfmx]{graphicx}
\usepackage{color}

\makeatletter
\@namedef{subjclassname@2010}{%
  \textup{2010} Mathematics Subject Classification}
\makeatother

\theoremstyle{plain} 
\newtheorem{thm}{Theorem}[section]
\newtheorem{cor}[thm]{Corollary}
\newtheorem{lem}[thm]{Lemma}
\newtheorem{prp}[thm]{Proposition}

\theoremstyle{definition}
\newtheorem{dfn}[thm]{Definition}
\newtheorem{exm}[thm]{Example}
\newtheorem{rmk}[thm]{Remark}

\newtheorem*{cnj}{Conjecture}
\newtheorem*{prp-n}{Proposition}
\newtheorem*{thm-n}{Theorem}
\newtheorem*{lem-n}{Lemma}

\numberwithin{equation}{section}

\def\al{\alpha}
\def\be{\beta}
\def\ga{\gamma}
\def\de{\delta}
\def\ep{\varepsilon}
\def\ze{\zeta}
\def\et{\eta}

\def\la{\lambda}
\def\ro{\rho}
\def\si{\sigma}

\def\ta{\tau}

\def\ph{\phi}
\def\ps{\psi}

\def\La{\Lambda}

\def\Om{\Omega}
\def\Ph{\Phi}
\def\Ps{\Psi}

\def\Ker{\operatorname{Ker}}

\def\Im{\operatorname{Im}}

\def\Aut{\operatorname{Aut}}
\def\End{\operatorname{End}}

\def\Mod{\operatorname{Mod}}

\def\calA{{\mathcal A}}
\def\calB{{\mathcal B}}
\def\calC{{\mathcal C}}

\def\calS{{\mathcal S}}

\def\calX{{\mathcal X}}

\def\bbZ{{\mathbb Z}}

\def\op{^{\mathrm{op}}} 

\def\inv{^{-1}}

\def\implies{\text{$\Rightarrow$}\ }
\def\impliedby{\text{$\Leftarrow$}\ }
\def\equivalent{\text{$\Leftrightarrow$}\ }
\def\incl{\hookrightarrow}
\def\iso{\cong}

\def\ovl{\overline}
\def\Ds{\bigoplus}

\def\dsm#1,#2..#3{\bigoplus_{{#1}={#2}}^{#3}}
\def\sm#1,#2..#3{\sum_{{#1}={#2}}^{#3}}

\def\id{1\kern-.25em{\text{{\rm l}}}} 

\def\isoto{\ \raise.8ex\hbox{$^{\sim}$}\kern-.7em\hbox{$\to$}\ }

\def\ang#1{{\langle #1 \rangle}}
\def\ya#1{\xrightarrow{#1}}
\def\ay#1{\xleftarrow{#1}}
\def\blank{\operatorname{-}}

\def\bg{%
\family{cmr}\size{20}{12pt}\selectfont}

\def\bigzerou{%
\smash{\lower1.7ex\hbox{\bg 0}}}

\def\bph{\bar{\ph}}
\def\bps{\bar{\ps}}

\numberwithin{equation}{section}
\numberwithin{figure}{section}

\def\GCat{G\text{-}\mathbf{Cat}}
\def\GGrCat{G\text{-}\mathbf{GrCat}}
\def\ZCat{\bbZ\text{-}\mathbf{Cat}}
\def\ZGrCat{\bbZ\text{-}\mathbf{GrCat}}

\def\k{\Bbbk}

\def\To{\Rightarrow}
\def\incl{\hookrightarrow}

\def\autz{\Aut^0(\hat{A})}
\def\upbl{^{(\blank)}}

\title[On isomorphisms of generalized multifold extensions]%
{On isomorphisms of generalized multifold extensions of algebras without nonzero oriented cycles}

\author{H.~Asashiba, M.~Kimura, K.~Nakashima, M.~Yoshiwaki}
\address{
H.~Asashiba,
Department of Mathematics,
Faculty of Science,
Shizuoka University,
836 Ohya, Suruga-ku,
Shizuoka, 422-8529, Japan
}
\email{asashiba.hideto@shizuoka.ac.jp} 

\address{
M.~Kimura,
c/o H.~Asashiba,
Department of Mathematics,
Faculty of Science,
Shizuoka University,
836 Ohya, Suruga-ku,
Shizuoka, 422-8529, Japan
}
\email{shirakawasanchi@gmail.com}

\address{
K.~Nakashima,
Department of Mathematics,
Faculty of Science,
Shizuoka University,
836 Ohya, Suruga-ku,
Shizuoka, 422-8529, Japan
}
\email{nakashima.ken@shizuoka.ac.jp} 

\address{
M.~Yoshiwaki,
Department of Mathematics,
Faculty of Science,
Shizuoka University,
836 Ohya, Suruga-ku,
Shizuoka, 422-8529, Japan
\\
Osaka City University Advanced Mathematical Institute,
3-3-138 Sugimoto, Sumiyoshi-ku, 
Osaka, 558-8585, Japan.
}
\email{yoshiwaki.michio@shizuoka.ac.jp} 

\thanks{
This work is partially supported by Grant-in-Aid for Scientific Research 25610003 and 25287001 from JSPS (Japan Society for the Promotion of Science), and by JST (Japan Science and Technology Agency) CREST Mathematics (15656429).
}

\keywords{repetitive categories, orbit categories, algebra isomorphisms, group actions, graded categories}

\makeatletter
\@namedef{subjclassname@2010}{%
  \textup{2010} Mathematics Subject Classification}
\makeatother
\subjclass[2010]{16W20, 16W22, 16W50}

\begin{document}
\maketitle
\begin{abstract}
Assume that a basic algebra $A$ over an algebraically closed field $\k$ with a basic set $A_0$ of primitive idempotents has the property that $eAe=\k$ for all $e \in A_0$.
Let $n$ be a nonzero integer, and $\ph$ and $\ps$ 
two automorphisms of the repetitive category $\hat{A}$ of $A$
with {\it jump} $n$ (namely, they send $A^{[0]}$ to $A^{[n]}$, where $A^{[i]}$ is the $i$-th copy of $A$ in $\hat{A}$ 
for all $i \in \bbZ$).
If $\ph$ and $\ps$ coincide on the objects and if there exists a map $\ro \colon A_0 \to \k$
such that $\ro_0(y)\ph_0(a)=\ps_0(a)\ro _0(x)$ for all morphisms $a\in A(x,y)$,
then the orbit categories $\hat{A}/\ang{\ph}$ and $\hat{A}/\ang{\ps}$ are isomorphic 
as $\bbZ$-graded categories.

\end{abstract}
\section*{Introduction}

Throughout this paper $G$ is a group, and we fix an algebraically closed field $\k$,
and all categories and functors are assumed to be $\k$-linear, and
we assume that all algebras are basic, connected, finite-dimensional $\k$-algebras.
We identify each algebra $A$ with a locally bounded category
whose object set $A_0$ is given by a fixed complete set of
orthogonal primitive idempotents of $A$.
Therefore we assume that $A \ne 0$ and that automorphisms of $A$  send
$A_0$ to $A_0$. 

For a category $\calC$ we denote by $\Aut(\calC)$ the group of automorphisms of $\calC$. 
A pair $(\calC ,X)$ of a category and a group homomorphism $X:G\to \Aut (\calC)$
(we write $X_\al :=X(\al )$) is called a category with $G$-action.
Let  $(\calC ,X)$ be a category with a $G$-action.
We do not assume that the $G$-action $X$ is free in general.
In that case, we take the definition of the orbit category $\calC/G$
in \cite{Asa}.
If $\calC$ is locally bounded and the $G$-action is free, then we take the classical definition of the orbit category $\calC/G$ as in \cite{Gab},
in which case $\calC/G$ is a basic category.
We make full use of the fact that equivalences between basic categories turn out to be isomorphisms (e.g., Theorem \ref{gme-iso}, \ref{thm:cyc-orbit}).

A {\em generalized multifold $($more precisely, $n$-fold\/$)$ extension} (\cite{AK}) of an algebra $A$
is a category of  the form $\La:= \hat{A}/\ang{\ph}$, where $\hat{A}$ is the repetitive category of $A$,
$\ph$ is an automorphism of $\hat{A}$ with {\em jump} $n$ for some integer $n$
 (see Definition \ref{dfn:repet-cat}(4) and Proposition \ref{jump}), and
$\hat{A}/\ang{\ph}:= (\hat{A}, \ph)/\bbZ$ is the orbit category of the category
$\hat{A}$ with a $\bbZ$-action given by $\ph$.
Note that the category $\La$ has a finite number of objects (i.e., $\La$ is an algebra)
if and only if $n \ne 0$.
The category $\La$ is called a {\em twisted multifold extension} (\cite{Asa02}) of $A$ and denoted by
$T_{\ph}^n(A)$
if the automorphism $\ph$ above has the form $\hat{\ro}\nu_A^n$, where
$\hat{\ro}$ is the automorphism of $\hat{A}$ naturally induced from $\ro$ and
$\nu _{A}$ is the Nakayama automorphism of $\hat{A}$.
An algebra $A$ is called a {\em piecewise hereditary} algebra
if $A$ is derived equivalent to a hereditary algebra $H$,
further in that case $A$ is said to be {\em of tree type}
if the ordinary quiver of $H$ is an oriented tree.

In \cite{AK}, 
we have given a derived equivalence classification of generalized multifold extension
algebras of piecewise hereditary algebras of tree type,
which extends the derived equivalence classification of twisted multifold extension algebras
of piecewise hereditary algebra of tree type
given in \cite{Asa02}.
A key point of the proof was to show that
the generalized multifold extension $\hat{A}/\ang{\ph}$
of a piecewise hereditary algebra $A$
by an automorphism $\ph$ of $\hat{A}$ with jump $0 \ne n\in \bbZ$ is
derived equivalent to the twisted multifold extension
$T_{\ph _0}^n(A)$ of $A$,
where $\ph _0:=(\id ^{[0]})^{-1}\ph\nu _{A}^{-n} \id ^{[0]}$ of $A$
(see Definition \ref{dfn:repet-cat}(3) for the definition of $\id ^{[0]}$).
For algebras $A$ in examples that we checked
the algebras $\hat{A}/\ang{\ph}$ and $T^n_{\ph_0}(A)$ were very similar
and seemed to be not only derived equivalent but also isomorphic.
This led us to the following conjecture.
\begin{cnj}
If $A$ is a piecewise hereditary algebra of tree type,
then the algebras $\hat{A}/\ang{\ph}$ and $T^n_{\ph_0}(A)$ would be isomorphic.
\end{cnj}

In this paper we will show that this conjecture is true as a direct consequence
of the following lemma and proposition.
We say that an algebra $A$ {\em has no nonzero oriented cycles} if
$eAe = \k$ for all primitive idempotents $e$ of $A$.
Then the following holds.

\begin{lem-n}[Lemma \ref{pwh-nocyc}]
If $A$ is a piecewise hereditary algebra of tree type,
then $A$ has no nonzero oriented cycles.
\end{lem-n}

Note that for each automorphism $\ph$ of $\hat{A}$ we have
$
\hat{A}/\ang{\ph}(x^{[i]}, y^{[j]}) \iso \Ds_{k\in \bbZ}\hat{A}(\ph^kx^{[i]}, y^{[j]})
$
for all objects $x^{[i]}, y^{[j]}$ of $\hat{A}$, which gives us a $\bbZ$-grading
of $\hat{A}/\ang{\ph}$ by setting
$(\hat{A}/\ang{\ph})^k(x^{[i]}, y^{[j]})$\linebreak[1]$ :=\hat{A}(\ph^kx^{[i]}, y^{[j]})$.
We always regard the orbit category $\hat{A}/\ang{\ph}$ as a $\bbZ$-graded
category by this grading.
Then we will show the following.

\begin{prp-n}[Corollary \ref{gen-twist}]
Assume that $A$ has no nonzero oriented cycles.
Let $\ph $ be an automorphism of $\hat{A}$ with jump $0 \ne n \in \bbZ$.
Then $\hat{A}/\ang{\ph }$ and $T_{\ph _0}^n(A)$ are isomorphic
as $\bbZ$-graded algebras.
\end{prp-n}

More generally, we give the following criterion.

\begin{thm-n}[cf.\ Theorem \ref{gme-iso}]
Assume that $A$ has no nonzero oriented cycles.
Let $\ph$ and $\ps$ be automorphisms of $\hat{A}$ with jump $n \in \bbZ$
that coincide on the objects of $\hat{A}$.
Then $\hat{A}/\ang{\ph}$ and $\hat{A}/\ang{\ps}$ are isomorphic
as $\bbZ$-graded algebras (equivalence of $\bbZ$-categories when $n = 0$) if 
there exists a map $\ro _0\colon A_0 \to \k^{\times}$ such that
$\ro _0(y)\ph _0(a)=\ps _0(a)\ro _0(x)$ for all morphisms $a\in A(x,y)$.
\end{thm-n}
\noindent
Essential point to prove this theorem is
to combine an idea in Saor\'in's paper \cite{Sa} and Theorem \ref{orbitcat-eq} taken from \cite{Asa}.

The paper is organized as follows.
After some preparations in section 1
we characterize automorphisms of repetitive categories with jump $0$ in section 2.
In section 3 we discuss on equivalences of orbit categories.
We note that we do not have to assume that actions of groups are free throughout this paper
because we use tools developed in \cite{Asa11, Asa}.
Therefore this paper extends the derived equivalence classification of generalized $n$-fold ($0\neq n\in \bbZ$) extensions
of piecewise hereditary algebras of tree type in \cite{AK} to that including the case of
$n=0$.
In section 4 we give a proof of the theorem above and discuss related topics.
In particular, we give another criterion (Theorem \ref{thm:cyc-orbit}) for isomorphisms (equivalences) of
orbit categories in terms of simple cycles in the quiver of an algebra.
Finally in section 5 we apply the results in the previous section to
piecewise hereditary algebras to have an affirmative answer (Corollary \ref{cor:answer-conj}) to the conjecture above.

\subsection*{Acknowledgments}

We would like to thank Junichi Miyachi for asking us the existing of examples
that gives a negative solution to the conjecture above,
Manuel Saor\'in for sending us his unpublished paper \cite{Sa}
(the published version of which is \cite{Sa02})
and Steffen Koenig for
informing us the proof of Lemmas \ref{lem:fgld-local} and \ref{pwh-nocyc}.

\section{Preliminaries}

Let $Q = (Q_0, Q_1, s, t)$ be a quiver, namely, $Q_0$ is the set of vertices,
$Q_1$ is the set of arrows, $s$ and $t$ are the maps sending each
arrow to its source and target, respectively.
Then for vertices $x, y \in Q_0$, we set
$Q_1(x, y):= \{\al \in Q_1 \mid s(\al) = x, t(\al) = y\}$.

For a category $R$ we denote by $R_0$ and $R_1$
the class of objects and morphisms of $R$, respectively.
A category $R$ is said to be {\em locally bounded}
if it satisfies the following:
\begin{itemize}
\item
Distinct objects of $R$ are not isomorphic;
\item
$R(x,x)$ is a local algebra for all $x \in R_0$;
\item
$R(x,y)$ is finite-dimensional for all $x,y \in R_0$; and
\item
The set $\{y \in R_0 \mid R(x, y) \ne 0 \text{ or }
R(y,x) \ne 0\}$ is finite for all $x \in R_0$.
\end{itemize}
A category is called {\em finite}
if it has only a finite number of objects.

\begin{rmk}
Note that any locally bounded category $R$ is presented as $R=\k Q/I$ with a unique quiver $Q$ and  an admissible ideal $I$ of $\k Q$.
For each $\al \in Q_1$ we denote by $\ovl{\al}$ the morphism $\al + I$ of $R$.
\end{rmk}

A pair $(A, E)$ of an algebra $A$ and 
a complete set $E:=\{e_1, \dots, e_n\}$
of orthogonal primitive idempotents of $A$
can be identified with a locally bounded and finite category
$R$ by the following correspondences.
Such a pair $(A, E)$ defines a category $R_{(A,E)}:=R$
as follows:
$R_0:=E$, $R(x, y):=yAx$ for all $x, y \in E$,
and the composition of $R$ is defined by the multiplication of $A$.
Then the category $R$ is locally bounded and finite.
Conversely, a locally bounded and finite category $R$ defines
such a pair $(A_{R}, E_{R})$ as follows:
$A_{R}:= \Ds_{x,y\in R_0}R(x, y)$ with the usual matrix
multiplication (regard each element of $A$ as
a matrix indexed by $R_0$), and
$E_{R}:=\{ (\id_x\de_{(i,j),(x,x)})_{i,j \in R_0} \mid x \in R_0\}$.
We always regard an algebra $A$
as a locally bounded and finite category
by fixing a complete set $A_0$ of orthogonal
primitive idempotents of $A$.
We say that a locally bounded category $R$ {\em has no nonzero oriented cycles} if
$R(x, x) = \k$ for all objects $x$ of $R$, which extends the corresponding notion for
algebras.

\begin{dfn}\label{dfn:repet-cat}
Let $A$ be a locally bounded category.

(1) The {\em repetitive category} $\hat{A}$ of $A$ is a $\k$-category defined as follows ($\hat{A}$ turns out to be locally bounded again):
\begin{itemize}
\item
$\hat{A}_0:=A_0 \times \bbZ =\{ x^{[i]}:=(x,i)\mid x\in A_0, i\in \bbZ \}$.
\item
$\hat{A}(x^{[i]},y^{[j]}):=\begin{cases}
                                  \{ f^{[i]}\mid f \in A(x,y)\} &\text{if\ } j=i,\\
                                  \{ \be^{[i]}\mid \be \in DA(y,x)\} &\text{if\ } j=i+1,\\
                                  0 &\text{otherwise},
                                 \end{cases}
\quad \text{for all $x^{[i]},y^{[j]}\in \hat{A}_0$.}$
\item
For each $x^{[i]},y^{[j]},z^{[k]}\in \hat{A}_0$
the composition $\hat{A}(y^{[j]},z^{[k]})\times \hat{A}(x^{[i]},y^{[j]})\to \hat{A}(x^{[i]},z^{[k]})$
 is given as follows.
\begin{itemize}
\item[(i)]If $i=j,j=k$, then this is the composition of $A$ $A(y,z)\times A(x,y)\to A(x,z)$.
\item[(ii)]If $i=j,j+1=k$, then this is given by the right $A$-module structure of $DA$: $DA(z,y)\times A(x,y)\to DA(z,x)$.
\item[(iii)]If $i+1=j,j=k$, then this is given by the left $A$-module structure of $DA$: $A(y,z)\times DA(y,x)\to DA(z,x)$.
\item[(iv)]Otherwise, the composition is zero.
\end{itemize}
\end{itemize}

(2) We define an automorphism $\nu_A$ of $\hat{A}$,
called the {\em Nakayama automorphism} of $\hat{A}$, by
$\nu_A(x^{[i]}):= x^{[i+1]}$, $\nu_A(f^{[i]}):= f^{[i+1]}$,
$\nu_A(\be^{[i]}):= \be^{[i+1]}$
for all $i \in \bbZ, x \in A_0, f \in A_1, \be \in \bigcup_{x,y\in A_0} DA(y,x)$.

(3) For each $i \in \bbZ$,
we denote by $A^{[i]}$ the full subcategory of $\hat{A}$ formed
by $x^{[i]}$ with $x \in A$, and by
$\id^{[i]} : A \isoto A^{[i]}, x \mapsto x^{[i]}$, an isomorphism of categories.
Further we define a bijection $\et_i\colon D(A(y,x)) \to \hat{A}(x^{[i]}, y^{[i+1]})$
by $\be \mapsto \be^{[i]}$ for all $x, y \in A_0$.

(4) For an integer $n$ we say that an automorphism $\ph$ of $\hat{A}$ has
a {\em jump} $n$ if $\ph(A^{[0]})=A^{[n]}$.
\end{dfn}

We cite the following \cite[Lemma 1.5, Proposition 1.6.]{AK}.

\begin{lem}\label{lem:com-Nak}
Let $A$ be an algebra.
Then the actions of $\ph \nu_A$ and $\nu_A\ph$
coincide on the objects of $\hat{A}$
for all $\ph \in \Aut (\hat{A})$.
\end{lem}

\begin{prp}\label{jump}
Let $A$ be an algebra, $n$ an integer,
and $\ph$ an automorphism of $\hat{A}$.
Then the following are equivalent:
\begin{enumerate}
\item
$\ph$ is an automorphism with jump $n$;
\item
$\ph(A^{[i]}) = A^{[i+n]}$ for some integer $i$;
\item
$\ph(A^{[j]}) = A^{[j+n]}$ for all integers $j$; and
\item
$\ph = \ph _L\nu_A^n$ for some automorphism $\ph _L$ of $\hat{A}$ with jump $0$.
\item
$\ph = \nu_A^n\ph _R$ for some automorphism $\ph _R$ of $\hat{A}$ with jump $0$.
\end{enumerate}
\end{prp}

By this Proposition, we have enough to consider automorphisms with jump $0$.
We cite the following from \cite[Lemma 2.3]{Asa99}.

\begin{lem}\label{repet-auto}
Let $\ps \colon A \to B$ be an isomorphism of locally bounded categories.
Denote by $\ps^y_x \colon A(y, x) \to B(\ps y, \ps x)$
the isomorphism defined by $\ps$ for all $x, y \in A$.
Define $\hat{\ps} \colon \hat{A} \to \hat{B}$ as follows.
\begin{itemize}
\item For each $x^{[i]} \in \hat{A}, \hat{\ps}(x^{[i]}):= (\ps x)^{[i]}$;
\item For each $f^{[i]} \in \hat{A}(x^{[i]}, y^{[i]}), \hat{\ps}(f^{[i]}):= (\ps f)^{[i]}$; and
\item For each $\be^{[i]} \in \hat{A}(x^{[i]}, y^{[i+1]}),
\hat{\ps}(\be^{[i]}):= (D((\ps^y_x)\inv)(\be))^{[i]} = (\be \circ (\ps^y_x)\inv)^{[i]}$.
\end{itemize}
Then
\begin{enumerate}
\item[(1)]
$\hat{\ps}$ is an isomorphism.
\item[(2)]
Given an isomorphism $\rho \colon \hat{A} \to \hat{B}$,
the following are equivalent.
	\begin{enumerate}
	\item[(a)]
	$\rho = \hat{\ps}$;
	\item[(b)]
	$\rho$ satisfies the following.
		\begin{enumerate}
		\item[(i)]
		$\rho \nu_A = \nu_B \rho$;
		\item[(ii)]
		$\rho(A^{[0]}) = A^{[0]}$;
		\item[(iii)] The diagram
		$$
		\begin{CD}
		A @>{\ps}>> B\\
		@V{\id^{[0]}}VV @VV{\id^{[0]}}V\\
		A^{[0]} @>>{\rho}> B^{[0]}
		\end{CD}
		$$
		is commutative; and
		\item[(iv)]
		$\rho(\be^{[0]}) = (\be \circ (\ps^y_x)\inv)^{[0]}$ for
		all $x, y \in A$ and all $\be \in DA(y, x)$.
		\end{enumerate}
	\end{enumerate}
\end{enumerate}
\end{lem}

\section{Automorphisms of the repetitive category with jump $0$}

Throughout this section $A$ is an algebra.
We set $\autz$ to be the group of all automorphisms of $\hat{A}$ with jump $0$.
Note that an algebra $A$ is regarded as an $A$-$A$-bimodule
$A(\blank, \blank) \colon A\op \times A  \to \Mod \k$ that sends
$(x, y) \in (A\op \times A)_0$ to $A(x, y)$.

\begin{dfn}
For each $\ph, \ps \in \Aut(A)$
we denote by ${}_\ps A_\ph$ the $A$-$A$-bimodule $A(\ph(\blank), \ps(\blank))$ defined by $(x, y) \mapsto A(\ph(x), \ps(y))$ for all $x, y \in A_0 \cup A_1$.
Then a morphism $f \colon A \to {}_\ps A_\ph$ of $A$-$A$-bimodules
is nothing but a family  $f = (f_{y,x})_{y,x}$ of linear maps $f_{y,x}\colon A(x,y) \to A(\ph(x), \ps(y))$ satisfying
\begin{equation}\label{eq:semilinear}
f_{y',x'}(a\al b)=\ps(a)f_{y,x}(\al)\ph(b)  \
\text{for each $y' \ay{a}y \ay{\al} x \ay{b} x'$ in $A$}.
\end{equation}
In the sequel we usually omit subscripts $x, y, x', y'$ above.
\end{dfn}

\begin{lem}Let $A$ be an algebra.
Then we can construct a bijection $\ze$ from $\autz$ to the set $\calA$ of families of pairs
$(\ph _i, \bar{\ph}_i)_{i\in \bbZ}$,
where $\ph_i$ is an automorphism of $A$
and $\bar{\ph}_i \colon A \to
{}_{\ph_i}A_{\ph_{i+1}}$ is an isomorphism of $A$-$A$-bimodules
for all $i\in \bbZ$.
\end{lem}

\begin{proof}
Let $\ph \in \autz$.
For each $i\in \bbZ$ we define $(\ph_i, \bph_i) \in \calA$ as follows.
Since $\ph$ has jump 0, $\ph$ restricts to an automorphism of $A^{[i]}$, and
we set $\ph _i:=(\id^{[i]})^{-1}\ph \id^{[i]}:A\to A$, which turns out to be
an automorphism of $A$. 
For each $x, y \in A_0$ we define $\bph_i \colon A(y,x) \to A(\ph_{i+1}(y), \ph_i(x))$
by the following commutative diagram
\begin{equation}\label{eq:dfn-phi-2}
\vcenter{\xymatrix{
D(A(y,x)) & D(A(\ph_{i+1}(y), \ph_i(x)))\\
\hat{A}(x^{[i]}, y^{[i+1]}) &  \hat{A}(\ph(x^{[i]}), \ph(y^{[i+1]})).
\ar"1,1";"1,2"^(0.4){D(\bph_i\inv)}
\ar"2,1";"2,2"_(0.4){\ph}
\ar"1,1";"2,1"_{\et_i}
\ar"1,2";"2,2"^{\et_i}
}}
\end{equation}
Then $\bar{\ph}_i \colon A \to {}_{\ph_i}A_{\ph_{i+1}}$ is an isomorphism of $A$-$A$-bimodules.
Indeed, since $\ph \colon \hat{A} \to \hat{A}$ is a functor,
it induces a morphism of bimodules $\hat{A}(\blank, \blank) \to\hat{A}(\blank, \blank)$
giving us the commutativity of the inner rectangle in the following commutative diagram.
$$
\xymatrix{
D(A(y',x')) &&& D(A(\ph_{i+1}(y'), \ph_i(x')))\\
&\hat{A}(x'^{[i]}, y'^{[i+1]}) &  \hat{A}(\ph(x'^{[i]}), \ph(y'^{[i+1]}))&\\
&\hat{A}(x^{[i]}, y^{[i+1]}) &  \hat{A}(\ph(x^{[i]}), \ph(y^{[i+1]}))&\\
D(A(y,x)) &&& D(A(\ph_{i+1}(y), \ph_i(x)))
\ar"2,2";"2,3"^(0.4)\ph
\ar"3,2";"3,3"_(0.4)\ph
\ar"2,2";"3,2"_{\hat{A}(a^{[i]}, b^{[i+1]})}
\ar"2,3";"3,3"^{\hat{A}(\ph(a^{[i]}), \ph(b^{[i+1]}))}
\ar"1,1";"1,4"^{D(\bph_i\inv)}
\ar"4,1";"4,4"_{D(\bph_i\inv)}
\ar"1,1";"4,1"_{D(A(b, a))}
\ar"1,4";"4,4"^{D(A(\ph_{i+1}(b), \ph_i(a)))}
\ar"1,1";"2,2"^{\et_i}
\ar"1,4";"2,3"_{\et_i}
\ar"4,1";"3,2"_{\et_i}
\ar"4,4";"3,3"^{\et_i}
}
$$
Hence the commutativity of the outer rectangle shows the equality
\eqref{eq:semilinear}.

Conversely, let $(\ph_i, \bph_i) \in \calA$.
We construct a $\ph \in \autz$ as follows.
For each $x^{[i]}$ with $x \in A_0, i \in \bbZ$
we set $\ph(x^{[i]}) := \ph_i(x)^{[i]}$.
Define $\ph$ on morphisms as follows.
For each $i$ define $\phi\colon \hat{A}(x^{[i]}, y^{[i]}) \to \hat{A}(\ph_i(x)^{[i]}, \ph_i(y)^{[i]})$ by the following commutative diagram:

$$
\xymatrix{
\hat{A}(x^{[i]}, y^{[i]}) &\hat{A}(\ph_i(x)^{[i]}, \ph_i(y)^{[i]})\\
A(x, y) & A(\ph_i(x), \ph_i(y)),
\ar"1,1";"1,2"^(0.4)\phi
\ar"2,1";"2,2"_(0.4){\ph_i}
\ar"2,1";"1,1"^{\id^{[i]}}
\ar"2,2";"1,2"_{\id^{[i]}}
}
$$
and define  $\phi\colon \hat{A}(x^{[i]}, y^{[i+1]}) \to \hat{A}(\ph_i(x)^{[i]}, \ph_{i+1}(y)^{[i+1]})$ by
$\bar{\ph}_i$
with \eqref{eq:dfn-phi-2}.
Namely,
$\ph :\hat{A}\to \hat{A}$ by
$\ph(a^{[i]}):=(\ph _i(a))^{[i]}$ ($a \in A(x,y)$)
and $\ph (\be^{[i]}):=(D(\bph _i^{-1})(\be))^{[i]}$ ($\be \in D(A(y,x))$)
for all $x, y \in A_0$.
Then it is easy to see that $\ph \in \autz$ and that
these constructions are inverses to each other.
\end{proof}

\begin{rmk} \label{rmk:comp}
In the above we identify $\ph$ with $\ze(\ph)=(\ph _i,\bph_i)_{i\in \bbZ}$. Then\\
(1) for another $\ps =(\ps_i,\bar{\ps}_i)_{i\in \bbZ} \in \autz$
we have  $\ph \circ \ps =(\ph_i \circ\ps_i,\bph_i \circ \bar{\ps}_i)_{i\in \bbZ}$; and \\
(2) for each $\si \in \Aut(A)$ we have $\hat{\si} = (\si, \si )_{i\in \bbZ}$.

In (1) note that the composite $\bph_i \circ \bar{\ps}_i$ of the second  entries
 is a bimodule map
 $$
 A \to {}_{\ph_i\ps_i}A_{\ph_{i+1}\ps_{i+1}}$$
given by
$
A(x, y) \ya{\bps_i} A(\ps_{i+1}x, \ps_{i}y) \ya{\bph_i} A(\ph_{i+1}\ps_{i+1}x, \ph_i\ps_{i}y)
$
for all $x, y \in A_0$.
\end{rmk}

\begin{dfn}\label{dfn-Ps}
We define a group homomorphism $\Ps:\autz \to \Aut (A)$
by $\Ps(\ph):=\ph _0$ for all $\ph \in \autz$.
\end{dfn}
\begin{rmk} \label{epim}
For each $\si \in \Aut(A)$ we have $\hat{\si}\in \autz$ and $\Ps(\hat{\si})=\si$  
by Lemma \ref{repet-auto}.
Thus $\Ps$ is a retraction, in particular, an epimorphism.
\end{rmk}

It is obvious by definition that
an automorphism $\ph \in\autz$ is in $\Ker \Ps$
if and only if $\ph_0$ is the identity of $A$.
Therefore to have more information on $\Ker\Ps$,
we need to know how to construct automorphisms of $\hat{A}$
from the identity of $A$.

\begin{dfn}\label{dfn:xi}
We define a map $\xi:(\k^\times)^{A_0}\to \Aut (A)$ as follows.
For each $\la =(\la (x))_{x\in A_0}\in (\k^\times)^{A_0}$
we set
$$\xi (\la)(x):=x\quad \text{for all }x \in A_0$$
and
for each morphism $x \ya{a} y$ in $A$ we define $\xi (\la)(a)$ by the commutativity of the diagram
$$
\xymatrix@C=40pt{
x & x\\
y & y
\ar"1,1";"1,2"^{\la(x)\id_x}
\ar"2,1";"2,2"_{\la(y)\id_y}
\ar"1,1";"2,1"_{\xi(\la)(a)}
\ar"1,2";"2,2"^{a}
}
$$
in $A$, i.e., 
\begin{equation}\label{eq:def-xi}
\xi (\la)(a):=\la (y)^{-1}\la (x)a.
\end{equation}
Then it is easy to see that $\xi(\la) \in \Aut(A)$ for all $\la \in (\k^\times)^{A_0}$,
and that $\xi$ is a group homomorphism.

Note that we can regard $\la$ as a natural isomorphism
$\la\colon \xi(\la) \To \id_A$.
\end{dfn}

\begin{lem}\label{constr-auto}
Regard $(\k^\times)^{\hat{A}_0}=((\k^\times)^{A_0})^\bbZ$ by the canonical isomorphism
$(\k^\times)^{\hat{A}_0}=(\k^\times)^{A_0\times\bbZ}\iso ((\k^\times)^{A_0})^\bbZ$.
Then for each $\la=(\la _i)_{i\in \bbZ}\in (\k^\times)^{\hat{A}_0}$ 
we define a family $(\ph _i,\bph _i)_{i\in \bbZ}$ of maps as follows.
\begin{equation}\label{eq:phi}
\ph _i:=\begin{cases}
\xi (\la _0\la _1\cdots \la _{i-1}) & \text{if } i>0\\
\id _A & \text{if } i=0\\
            \xi (\la _i\la _{i+1}\cdots \la _{-1})\inv & \text{if } i<0
\end{cases} 
\end{equation}
and $\bph _i:A\to A$ is defined by 
\begin{equation}\label{eq:phibar}
\bph _i(a):=\la _i(s(a))\ph _i(a)\ (=\la _i(t(a))\ph _{i+1}(a))\text{ for $a\in A_1$.}
\end{equation}
Then for each $i \in \bbZ$ we have
\begin{enumerate}
\item
$\ph_i$ is an automorphism of $A$; and\\[-10pt]
\item 
$\bph_i\colon A \to{}_{\ph_i}A_{\ph_{i+1}}$ is an isomorphism of $A$-$A$-bimodule.
\end{enumerate}
\end{lem}
\begin{proof}
 (1) This is obvious because in \eqref{eq:phi} $\ph_i$ is  given by the left multiplication by a nonzero element of $\k$.  
 
 (2) In \eqref{eq:phibar} the equality $\la _i(s(a))\ph _i(a)=\la _i(t(a))\ph _{i+1}(a)$ 
follows from the diagram
$$
\xymatrix@C=36pt{
x&x&\cdots&x& x & x & \cdots & x &x\\
y&y&\cdots&y&y & y & \cdots & y&y,
\ar"1,1";"1,2"^{\la_i(x)\id_x}
\ar"2,1";"2,2"_{\la_i(y)\id_y}
\ar"1,1";"2,1"_{\ph_{i+1}(a)}
\ar"1,2";"2,2"^{\ph_i(a)}
\ar"1,2";"1,3"^{\la_{i-1}(x)\id_x}
\ar"2,2";"2,3"_{\la_{i-1}(y)\id_y}
\ar"1,3";"1,4"^{\la_1(x)\id_x}
\ar"2,3";"2,4"_{\la_1(y)\id_y}
\ar"1,4";"2,4"_{\ph_1(a)}
\ar"1,4";"1,5"^{\la_0(x)\id_x}
\ar"2,4";"2,5"_{\la_0(y)\id_y}
\ar"1,5";"2,5"_{a=\ph_0(a)}
\ar"1,6";"1,5"_{\la_{-1}(x)\inv\id_x}
\ar"2,6";"2,5"^{\la_{-1}(y)\inv\id_y}
\ar"1,6";"2,6"_{\ph_{-1}(a)}
\ar"1,7";"1,6"_{\la_{-2}(x)\inv\id_x}
\ar"2,7";"2,6"^{\la_{-2}(y)\inv\id_y}
\ar"1,8";"1,7"_{\la_{j+1}(x)\inv\id_x}
\ar"2,8";"2,7"^{\la_{j+1}(y)\inv\id_y}
\ar"1,8";"2,8"^{\ph_{j+1}(a)}
\ar"1,9";"1,8"_{\la_j(x)\inv\id_x}
\ar"2,9";"2,8"^{\la_j(y)\inv\id_y}
\ar"1,9";"2,9"^{\ph_j(a)}
}
$$
the commutativity of which follows from \eqref{eq:phi}.
[In fact conversely the condition \eqref{eq:phi} is determined so that this diagram commutes.]

We show $\bar{\ph}_i \colon A \to {}_{\ph_i}A_{\ph_{i+1}}$ is an isomorphism of $A$-$A$-bimodules.
Indeed, by definition $\bar{\ph}_i$ is a bijection and for each $\al ,\be ,\ga \in A_1$ we have 
\begin{eqnarray*}
\bph _i(\al \be \ga )&=&\la _i(s(\ga ))\ph _i(\al \be \ga )\\
         &=&\ph _i(\al )\la _i(s(\ga ))\ph _i(\be \ga )\\
         &=&\ph _i(\al )\bph _i(\be \ga )\\
         &=&\ph _i(\al )\la _i(t(\be ))\ph _{i+1}(\be \ga )\\
         &=&\ph _i(\al )\la _i(t(\be ))\ph _{i+1}(\be )\ph _{i+1}(\ga )\\
         &=&\ph _i(\al )\bph _i(\be )\ph _{i+1}(\ga ).
\end{eqnarray*}
\end{proof}

\begin{cor}
By Lemma \ref{constr-auto} define a map
$\Ph\colon (\k^\times)^{\hat{A}_0} \to \autz$
by $\Ph(\la):= (\ph_i,\bph_i)_{i\in \bbZ}$. Then $\Phi$ is a group homomorphism. 
\end{cor}
\begin{proof}
Take another $\mu \in (\k^\times)^{\hat{A}_0}$ with $\Ph(\mu):= (\ps_i,\bar{\ps}_i)_{i\in \bbZ}$. 
Set $\Ph (\la \circ \mu) = (\chi_i,\bar{\chi}_i)_{i\in \bbZ}$.
By Remark \ref{rmk:comp}, 
we have $\Ph(\la)\circ\Ph(\mu)=(\ph_i\circ\ps_i,\bph_i\circ\bar{\ps}_i)_{i\in \bbZ}$.
Therefore, to verify that  $\Ph(\la\mu)=\Ph(\la)\circ\Ph(\mu)$
it is enough to show that 
$$
\left\{
\begin{aligned}
\chi_i &=&\ph_i \circ \ps_i \\
\bar{\chi}_i &=&\bph_i \circ \bar{\ps}_i . \\
\end{aligned}
\right.
$$
By the definition of the multiplication in $(\k^\times)^{\hat{A}_0}$
the first equality is obvious because $\xi$ is a group homomorphism.
To show the second equality let $a\colon x \to y$ be in $A_1$.
Then 
$$
\bar{\chi}_i (a) = (\la \mu)_i (x)  \chi_i (a) 
= \la_i (x) \mu_i (x) \ph_i (a) \ps_i (a) 
=\bph_i (a) \bar{\ps}_i (a)
=(\bph_i  \circ \bar{\ps}_i ) (a),
$$
as required. 
\end{proof}

We assume the following property
which is necessary for our purpose.

\begin{dfn}\label{dfn-nocyc}
A locally bounded category $R$ is said to {\em have no nonzero oriented cycles}
if $R(x, x) \iso \k$ for all objects $x$ of $R$.
\end{dfn}

\begin{rmk}\label{rmk:rpt-noc}
(1) Note that if a locally bounded category $R$ is presented as $R=\k Q/I$ with  a quiver $Q$ and  an admissible ideal $I$ of $\k Q$, 
then $R$ has no nonzero oriented cycles if and only if  $I$ contains all oriented cycles in $Q$.

(2) If a locally bounded category $A$ has no nonzero oriented cycles, then
so does the repetitive category $\hat{A}$ because for any object $x^{[i]}$ of $\hat{A}$
we have $\hat{A}(x^{[i]}, x^{[i]})\iso A(x, x)\iso \k$.
\end{rmk}

\begin{prp}\label{rigid-exact}
Assume that $A$ has no nonzero oriented cycles.
Then there is an exact sequence of groups
$$1\to (\k^\times)^{\hat{A}_0} \xrightarrow{\Ph} \autz \xrightarrow{\Ps} \Aut (A) \to 1.$$
\end{prp}
\begin{proof}
First, by Remark \ref{epim} we already know that $\Ps$ is an epimorphism.
Second, we show that $\Ph$ is injective.
Let $\la \in \Ker \Ph$ and set $\Ph(\la) = (\ph_i, \bph_i)_{i\in \bbZ}$.
Then $(\ph_i, \bph_i)_{i\in \bbZ} = \Ph(\la )=\id _{\hat{A}} = (\id_A, \id_A)_{i\in \bbZ}$.
To show that $\la_i(x)=1$ for all $x\in A$ and $i\in \bbZ$
we show (i) for each $i\in \bbZ$, $\la_i(x)=\la_i(y)=:k_i \in \k$ for all $x,y\in A$;
and (ii) for each $i\in \bbZ$, $k_i=1$. 

To show the statement (i) we first show that $\xi(\la_i) = \id_A$ for all $i \in \bbZ$.
For each $i > 0$, 
$$\id_A = \ph_i = \xi (\la _0\la _1\cdots \la _{i-1})
= \xi (\la _0\la _1\cdots \la _{i-2})\xi (\la _{i-1})
=\ph_{i-1}\xi (\la _{i-1})
=\xi (\la _{i-1}).
$$
For each $i < 0$,
$$
\begin{aligned}
\id_A &= \ph_i =\xi (\la _i\la _{i+1}\cdots \la _{-1})\inv
=(\xi (\la _i)\xi(\la _{i+1}\cdots \la _{-1}))\inv\\
&=\xi(\la _{i+1}\cdots \la _{-1})\inv \xi (\la _i)\inv
=\ph_{i+1}\xi (\la _i)\inv= \xi (\la _i)\inv,
\end{aligned}
$$
thus $\xi(\la_i) = \id_A$.
Accordingly, for each morphism $a\colon x \to y$ in $A$,
we have $a = \xi(\la_i)(a) = \la_i(y)\inv \la_i(x) a$, and hence
$\la_i(x) = \la_i(y)$ for all $i \in \bbZ$.
Since $A$ is connected, there exists some $k_i \in \k^{\times}$ such that $\la_i(x) = k_i$
for all $x \in A_0$.

Now let $a \colon x \to y$ be a morphism in $A$.
Then for each $i \in \bbZ$, $\bph_i(a) = \la_i(x) \ph_i(a) = k_i a$.
Therefore $\id_A=\bph_i = k_i \id_A$, thus $k_i = 1$.
This shows that $\Ph$ is injective.

Next we show that $\Im \Ph =\Ker \Ps$.
Let $\la \in (\k^\times)^{\hat{A}_0}$ and set $\Ph(\la) =(\ph_i,\bph _i)_{i\in \bbZ}$.
Then $\Ps\Ph(\la)=\ph_0=\id_A$ by definition.
Hence $\Im \Ph \subseteq\Ker \Ps$.
Conversely we show that $\Im \Ph \supseteq \Ker \Ps$.
Let $\ps =(\ps _i,\bps _i)_{i\in \bbZ}\in \Ker \Ps$. 
Then $\ps_0=\id_A$. 
By Lemma \ref{lem:com-Nak} we see that $\ps_i=\id_A$ for all $i\in\bbZ$ on the objects.
Since $\bar{\ps}_i \colon A \to {}_{\ps_i}A_{\ps_{i+1}}$ is an isomorphism of $A$-$A$-bimodules, 
$\bps_i(\id_x)\in A(\ps_{i+1}(x),\ps_i(x))=A(x,x)=\k\id_x$
because $A$ has no nonzero oriented cycles. 
Therefore for each $i\in \bbZ$ and $x\in A_0$ there exists $\la _i(x)\in \k^\times$ such that $\bps _i(\id_x)=\la _i(x)\id_x$.
Set $\la:= (\la_i(x))_{(x,i)\in A_0\times\bbZ} \in (\k^{\times})^{\hat{A}_0}$.
It remains to show that $\ps =\Ph(\la)$.
Set $\Ph(\la)=(\ph _i,\bph _i)_{i\in \bbZ}$. 
Then we have only to show
$\ps_i=\ph_i$ and  $\bps_i=\bph_i$  for all $i\in\bbZ$. 

It  follows from $\ps _0=\id _A=\ph _0$ that $\ps_i=\ph_i\  (=\id_A)$ on the objects  for all $i\in\bbZ$ by Lemma \ref{lem:com-Nak}.

For each morphism $a\in A(x,y)\ (x,y\in A)$ we show $\ps_i(a)=\ph_i(a)$ by induction on  $i\ge 1$.
\begin{eqnarray*}
\ph _i(a)&=&\xi(\la _0\la _1\cdots \la _{i-1})(a)\\
         &=&\la _0\cdots \la _{i-1}(x)(\la _0\cdots \la _{i-1}(y))^{-1}a\\
         &=&\la _{i-1}(x)(\la _{i-1}(y))^{-1}\ph_{i-1}(a)\\
         &=&\la _{i-1}(x)(\la _{i-1}(y))^{-1}\ps_{i-1}(a)\\
         &=&(\la _{i-1}(y))^{-1}\ps_{i-1}(a)\la_{i-1}(x)\id_x\\
         &=&(\la _{i-1}(y))^{-1}\ps _{i-1}(a)\bps _{i-1}(\id_x)\\
         &=&(\la _{i-1}(y))^{-1}\bps _{i-1}(a)\\
         &=&(\la _{i-1}(y))^{-1}\bps  _{i-1}(\id_y a)\\
         &=&(\la _{i-1}(y))^{-1}\bps  _{i-1}(\id_y)\ps _i(a)\\
         &=&(\la _{i-1}(y))^{-1}\la _{i-1}(y)\id_y\ps _i(a)\\
         &=&\ps _i(a)
\end{eqnarray*}
Similarly we see that  $\ps_i(a)=\ph_i(a)$ for all  $i < 0$.
Let  $a\in A(x,y)$ with $x,y\in A$.
Then
\begin{eqnarray*}
\bph _i(a)&=&\la _i(x)\ph _i(a)\\
      &=&\la _i(x)\ps _i(a)\\
      &=&\la _i(x)\ps _i(a\id_x)\\
      &=&\ps _i(a)\la _i(x)\ps _i(\id_x)\\
      &=&\ps _i(a)\la _i(x)\id_{\ps_i(x)}\\
     &=&\ps _i(a)\la _i(x)\id_x\\
      &=&\ps _i(a)\bps _i(\id_x)\\
      &=&\bps _i(a)
\end{eqnarray*}
for all $i\in \bbZ$. 
\end{proof}

\begin{rmk}\label{jump-rmk}
\hspace{1cm}
\begin{enumerate}
\item[(1)]By Remark \ref{epim}, the exact sequence in Proposition \ref{rigid-exact} splits.
Therefore we have an isomorphism 
$$
(\k^\times)^{\hat{A}_0} \times \Aut (A) \to \autz
$$
sending $(\la, \si)$ to $\hat{\si} \circ \Ph (\la)   =(\si  \ph_i, \si \bph_i )_{i\in \bbZ}$ by Remark \ref{rmk:comp} (2), 
where $ \Ph (\la) =(\ph_i , \bph_i)_{i\in \bbZ}$. 

Let $\ps =(\ps _i,\bps _i)_{i\in \bbZ}$ be an automorphism of $\hat{A}$ with jump 0.
By the above, we have $\ps =(\si \ph_i, \si \bph_i  )_{i\in \bbZ}$ 
for some $(\la, \si) \in (\k^\times)^{\hat{A}_0} \times \Aut (A)$.
By the equality \eqref{eq:phibar}, 
$\la _i(x)\si \ph _i(a) =\la _i(y)\si \ph _{i+1}(a) $ for all morphism $a\in A(x,y)$.
Therefore, we have
\begin{equation} \label{eq:ind-step}
\ps _{i+1}(a)=\la _i(x)(\la _i(y))^{-1}\ps _i(a). 
\end{equation}
\item[(2)]In Ohnuki-Takeda-Yamagata \cite[section 3]{OTY} 
they gave a surjection 
$$
\operatorname{U}(A)^\bbZ \times \operatorname{aAut} (A) \to \operatorname{aAut}_0 (\hat{A}),
$$
where $\operatorname{aAut} (B) $ denotes the group of algebra automorphisms of an algebra $B$ (here $\hat{A}$ is regarded as an algebra without identity) 
and  $\operatorname{U}(A)$ is the set of all units in $A$.
\end{enumerate}
\end{rmk}

\section{Orbit categories}

We cite some statements from \cite{Asa11, Asa} below.
\begin{dfn}
Let $(\calC ,X)$, $(\calC ',X')$ be categories with $G$-actions.
Then a $G$-{\em equivariant} functor is a pair $(F,\et)$
of a functor $F\colon \calC \to \calC'$ and a family $\et =(\et _\al )_{\al \in G}$
of natural isomorphisms $\et _\al :X'_\al F\Rightarrow FX_\al$ ($\al \in G$)
such that the following diagram commutes for all $\al ,\be \in G$
\[ \xymatrix{ X'_{\be\al }F=X'_\be X'_\al F & X'_\be FX_\al \\
                                         & FX_{\be \al}=FX_\be X_\al.
              \ar_{\et _{\be \al}}"1,1";"2,2"
              \ar^(0.6){X'_\be \et _\al}"1,1";"1,2"
              \ar^{\et _\be X_\al}"1,2";"2,2"}
\]
\end{dfn}

\begin{dfn}
Let $(E,\rho),(E',\rho')\colon\mathcal{C}\to\mathcal{C}'$ be $G$-equivariant
functors. Then a \emph{morphism} from $(E,\rho)$ to $(E',\rho')$
is a natural transformation $\eta\colon E\To E'$ such that
the diagrams
$$
\xymatrix{
A_{a}E & EA_a\\
A_aE' &E'A_a
\ar@{=>}^{\ro_a}"1,1";"1,2"
\ar@{=>}_{\ro'_a}"2,1";"2,2"
\ar@{=>}_{A_a\et}"1,1";"2,1"
\ar@{=>}^{\et A_a}"1,2";"2,2"
}
$$
commute for all $a\in G$.
\end{dfn}

\begin{lem}
\label{lem:eqv-eqv-is-eqv}Let $ $$\xymatrix{\mathcal{C}\ar[r]^{(E,\rho)} & \mathcal{C}'\ar[r]^{(E',\rho')} & \calC''}
$ $ $be $G$-equivariant functors of $G$-categories. Then

$(1)$ $(E'E,((E'\rho_{a})(\rho_{a}'E))_{a\in G})\colon\mathcal{C}\to\mathcal{C}''$
is a $G$-equivariant functor, which we define to be the composite
$(E',\rho')(E,\rho)$ of $(E,\rho)$ and $(E',\rho')$.

$(2)$ If further $(E'',\rho'')\colon\mathcal{C}''\to\mathcal{C}'''$
is a $G$-equivariant functor, then we have \[
((E,\rho)(E',\rho'))(E'',\rho'')=(E,\rho)((E',\rho')(E'',\rho'')).\]
\end{lem}

\begin{dfn}\label{dfn:G-Cat}
A 2-category $G$-$\mathbf{Cat}$ is defined as follows.
\begin{itemize}
\item The objects are the small $G$-categories.
\item The 1-morphisms are the $G$-equivariant functors between objects.
\item The identity 1-morphism of an object $\calC$ is the 1-morphism
$(\id_\calC, (\id_{\id_\calC})_{a \in G})$.
\item The 2-morphisms are the morphisms of $G$-equivariant functors.
\item The identity 2-morphism of a 1-morphism $(E, \ro) \colon \calC \to \calC'$
is the identity natural transformation  $\id_E$ of $E$, which is clearly a 2-morphism.
\item The composition of 1-morphisms is the one given in the previous lemma.
\item The vertical and the horizontal compositions of 2-morphisms are given
by the usual ones of natural transformations.
\end{itemize}
\end{dfn}

\begin{thm}\label{eqvar-eq}
Let $(E,\rho)\colon\mathcal{C}\to\mathcal{C}'$ be a $G$-equivariant functor in $\GCat$.
Then the following are equivalent.
\begin{enumerate}
\item $(E, \rho)$ is an equivalence in $\GCat$;
\item $E$ is fully faithful and dense
$($i.e., $E$ is a category equivalence$)$.
\end{enumerate}
Thus the {\em $G$-equivariant equivalences} 
are exactly the equivalences in $\GCat$.
\end{thm}

\begin{dfn}\label{dfn-graded-cat}
$(1)$ A $G$-\emph{graded} category is a category $\mathcal{B}$
together with a family of direct sum decompositions $\mathcal{B}(x,y)=\bigoplus_{a\in G}\mathcal{B}^{a}(x,y)$
$(x,y\in\mathcal{B})$ of $\Bbbk$-modules such that $\mathcal{B}^{b}(y,z)\cdot\mathcal{B}^{a}(x,y)\subseteq\mathcal{B}^{ba}(x,z)$
for all $x,y\in\mathcal{B}$ and $a,b\in G$. If $f\in\mathcal{B}^{a}(x,y)$
for some $a\in G$, then we say that $f$ is {\em homogeneous}
of {\em degree} $a$.

$(2)$ A \emph{degree-preserving} functor is a pair $(H,r)$ of a
functor $H\colon\mathcal{B}\to\mathcal{A}$ of $G$-graded categories
and a map $r\colon \mathcal{B}_0\to G$ such that $H(\mathcal{B}^{r_{y}a}(x,y))\subseteq\mathcal{A}{}^{ar_{x}}(Hx,Hy)$
for all $x,y\in\mathcal{B}$ and $a\in G$.
This $r$ is called a \emph{degree adjuster} of $H$.

$(3)$ A functor $H\colon\mathcal{B}\to\mathcal{A}$ of $G$-graded
categories is called a \emph{strictly} degree-preserving functor if
$(H,1)$ is a degree-preserving functor, where 1 denotes the constant
map $\mathcal{B}_0\to G$ with value $1\in G$, i.e., if $H(\mathcal{B}^{a}(x,y))\subseteq\mathcal{A}^{a}(Hx,Hy)$
for all $x,y\in\mathcal{B}$ and $a\in G$.

$(4)$
 A functor $H \colon \calB \to  \calA$  of $G$-graded categories
is said to be {\em homogeneously dense}
if for each $x \in \calA_0$ there exists some $y \in \calB_0$ such that
there exists a homogeneous isomorphism $x \to H(y)$.

$(5)$ Let $(H,r),(I,s)\colon\mathcal{B}\to\mathcal{A}$ be degree-preserving
functors. Then a natural transformation $\theta\colon H\To I$ is
called a \emph{morphism} of degree-preserving functors if $\theta x\in\mathcal{A}{}^{s_{x}^{-1}r_{x}}(Hx,Ix)$
for all $x\in\mathcal{B}$.

\end{dfn}

The composite of degree-preserving functors can be made into again
a degree-preserving functor as follows.
\begin{lem}
Let ${}$$\xymatrix{\mathcal{B}\ar[r]^{(H,r)} & \mathcal{B}'\ar[r]^{(H',r')} & \mathcal{B}''}
$ be degree-preserving functors. Then \[
(H'H,(r_{x}r'_{Hx})_{x\in\mathcal{B}})\colon\mathcal{B}\to\mathcal{B}''\]
 is also a degree-preserving functor, which we define to be the \emph{composite}
$(H',r')(H,r)$ of $(H,r)$ and $(H',r')$.\end{lem}

\begin{dfn}\label{dfn:G-GrCat}
A 2-category $G$-$\mathbf{GrCat}$ is defined as follows.
\begin{itemize}
\item The objects are the small $G$-graded categories.
\item The 1-morphisms are the degree-preserving functors between objects.
\item
The identity 1-morphism of an object $\calB$ is the 1-morphism $(\id_\calB, 1)$.
\item The 2-morphisms are the morphisms of degree-preserving functors.
\item
The identity 2-morphism of a 1-morphism $(H, r)\colon \calB \to \calA$
is the identity natural transformation $\id_H$ of $H$, which is a 2-morphism
(because $(\id_{H})x =\id_{Hx} \in \calA^1(Hx, Hx) = \calA^{r_x\inv r_x}(Hx, Hx)$
 for all $x \in \calB$).
\item The composition of 1-morphisms is the one given in the previous lemma.
\item The vertical and the horizontal compositions of 2-morphisms are given
by the usual ones of natural transformations.
\end{itemize}
\end{dfn}

\begin{thm}\label{eq-GrCat}
Let $(H,r) \colon\mathcal{B}\to\mathcal{A}$ be a degree-preserving functor in $\GGrCat$.
Then the following are equivalent.
\begin{enumerate}
\item $(H,r)$ is an equivalence in $\GGrCat$.
\item $H$ is fully faithful and homogeneously dense.
\end{enumerate}
\end{thm}

The following statement follows from
\cite[Definition 7.1, Theorem 7.5, Theorem 9.1 and Theorem 9.5]{Asa}

\begin{thm}\label{orbitcat-eq}
Let $(\calC ,X)$, $(\calC ',X')$ be categories with $G$-actions.  
Then the following are equivalent.
\begin{enumerate}
\item
There exists a $G$-equivariant equivalence $\calC \to \calC '$.
\item
There exists a strictly degree preserving, homogeneously dense equivalence
$\calC /G \to \calC '/G$ of $G$-graded categories.
\item
There exists a degree preserving, homogeneously dense equivalence
$\calC /G \to \calC '/G$ of $G$-graded categories.
\end{enumerate}
\end{thm}

\begin{rmk}
If $\calC'/G$ has the property that $(\calC'/G)(x, x)$ is local for all
$x \in (\calC'/G)_0$,
e.g., if $\calC'$ is locally bounded, then all equivalences $\calC/G \to \calC'/G$
are automatically homogeneously dense by \cite[Example 9.4(1)]{Asa}.
\end{rmk}

\begin{dfn}\label{dfn-walk}
Let $Q=(Q_0,Q_1,s,t)$ be a quiver.

$(1)$ The quiver $\overline{Q}=(Q_0,Q_1\sqcup Q'_1,\ovl{s},\ovl{t})$ is called the \emph{double quiver} of $Q$,
where $Q'_1:=\{a\inv \mid a \in Q_1\}$
and $\ovl{s}(a):=s(a), \ovl{s}(a\inv):=t(a), \ovl{t}(a):=t(a),\ovl{t}(a\inv):=s(a)$ for each $a \in Q_1$.
A path in $\ovl{Q}$ is called a {\em walk} in $Q$.
The {\em length} $|w|$ of a walk $w$ is the length of $w$ as a path in $\ovl{Q}$.

$(2)$ A walk $C=l_n \dots l_1$
$(l_1, \dots, l_n\in \ovl{Q}_1)$ in $Q$
is called a \emph{cycle} in $Q$ if $s(l_1)=t(l_n)$.
The cycle $C$ is called {\em simple} if $s(l_i)\neq s(l_j)$ for all $i\neq j$.

$(3)$ Let $C = l_m\cdots l_1$ and $C'=l'_n\cdots l'_1$ be simple cycles in $Q$
and $c\in\{1,\dots, m\}$.
Then we set $C[c]:= l_{\ovl{m+c}}\ \cdots \ l_{\ovl{1+c}}$,
where $\ovl{j}$ is the number in $\{1, \dots, m\}$ such that
$\ovl{j} \equiv j \pmod{m}$ for all integers $j$. 
We say that $C$ and $C'$ are \emph{equivalent} if $m=n$
and there exists some $c\in\{1,\dots,m\}$
such that  $C[c]= C'$ or ${C'}\inv$,

(4) Let $(\ph_a)_{a\in Q_1} \in  (\k^\times)^{Q_1}$ be a sequence and 
$C = l_n\cdots l_1$ be a walk in $Q$.
Then we set $\ph_C:=\ph_{l_n}\dots\ph_{l_{1}}$,
where $\ph_{a\inv}:=\ph^{-1}_a$ for each $a \in Q_1$.
\end{dfn}

\begin{prp}\label{tekito}
Let $Q$ be a quiver and $(\ph_a)_{a\in Q_1}, (\ps_a)_{a\in Q_1} \in  (\k^\times)^{Q_1}$ be
sequences.

$(1)$ We have $\ph_C=\ph_{C[c]}$ for all integers $c$ with $1 \le c \le |C|$.

$(2)$
We fix a complete set $\calS$ of representatives of equivalence classes of simple cycles in $Q$.
Then $\ph_{S}=\ps_{S}$ for all $S \in\calS$ if and only if
$\ph_{C}=\ps_{C}$ for all cycles $C$ in $Q$.
\end{prp}

\begin{proof}
$(1)$ Obvious.

$(2)$ (\impliedby\!). This is trivial.

(\implies\!\!).
Assume that $\ph_{S}=\ps_{S}$ for all $S \in\calS$.
Then obviously we have $\ph_{T}=\ps_{T}$ for all simple cycles $T$.
Let $C=l_n\dots l_{1}$ be a cycle in $Q$.
We show the statement by induction on the length $n$ of $C$.
When $n = 1$, the claim holds because $C$ is a simple cycle.
Assume $n > 1$.
We set $x_i:= s(l_i)$ for all $i = 1,\dots, n$.
Note that there exists an $i$ such that
$\{x_1\}  \not \ni x_2, \{x_1, x_2\} \not\ni x_3, \dots, \{x_1,\dots, x_{i-2}\} \not\ni x_{i-1},
\{x_1,\dots, x_{i-1}\} \ni x_{j} = x_i$.
Then $S:=l_{i-1}\dots l_{j}$ is a simple cycle and we have $\ph_{S}=\ps_{S}$.
We set $C' :=  l_{j-1}\cdots l_2 l_1\cdots l_{i+1}l_i$.
Then $C[j-1] =C' S$.
By induction hypothesis we have $\ph_{C'} = \ps_{C'}$.
Hence
$$\ph_C = \ph_{C[j-1]} =\ph_{C'}\ph_S = \ps_{C'}\ps_S = \ps_{C[j-1]} =\ps_C.$$
\end{proof}

For each $g \in \Aut(\calC)$ 
denote by $g\upbl$ the $\bbZ$-action on $\calC$
defined by $n \mapsto g^n\ (n \in \bbZ)$.
The following plays a central role in the proof of the main result.

\begin{prp}\label{orbitcat-iso}
Let $R$ be a locally bounded category, $Q$ the ordinary quiver of $R$,
and $g, h$ automorphisms
of $R$ that coincide on the objects of $R$.
Consider the following two statements.

\begin{enumerate}
\item
There exists a map $\ro \colon R_0 \to \k^{\times}$ such that
$\ro(y)g(f) = h(f)\ro(x)$ for all morphisms $f\in R(x,y)$ and for all $x,y \in R_0$.
\item
\begin{enumerate}
\item
For each $x, y \in Q_0$ and each $\al \in Q_1(x,y)$,
there exists some $c_{x,y} \in \k^\times$
such that
$(g\inv h)(\ovl{\al}) = c_{x,y}\ovl{\al}$, and
\item
for each cycle $C = l_n\cdots l_1$ $(l_1, \dots, l_n \in \ovl{Q}_1)$
in $Q$, we have 
$(g\inv h)_C = 1$, where $(g\inv h)_C := (g\inv h)_{l_n}\cdots (g\inv h)_{l_1}$, and 
 $(g\inv h)_{\al}:=c_{s(\al),t(\al)}$,
 $(g\inv h)_{\al\inv}:=(c_{s(\al),t(\al)})\inv$ $ (\al\in Q_1)$.
\end{enumerate}
\item
There exists a natural isomorphism $\et\colon h \To g$ such that
 $(\id_R, \et)\colon (R,g\upbl) \to (R,h\upbl)$
 $\bbZ$-equivariant equivalence.
 \item
 There exists a natural isomorphism $\et\colon h \To g$ such that 
 $(\id_R, \et)/\bbZ \colon R/\ang{g} \to R/\ang{h}$ is a strictly degree preserving, homogeneously dense equivalence of $\bbZ$-graded categories.
\end{enumerate}
Then $(1)$ is equivalent to $(2)$, 
$(3)$ is equivalent to $(4)$, and $(1)$ implies $(3)$.
If $R$ has no nonzero oriented cycles, then $(3)$ implies $(1)$.
\end{prp}

\begin{proof}
We set $E:=g\inv h$.

(1) \implies (2).
By (1) we have $E(f)=\ro(x)\inv\ro(y)f$ for all $f \in R(x,y)$ and for all $x,y \in R_0$.
We set $c_{x,y}:=\ro(x)\inv\ro(y)$ for all  $ x,y \in Q_0$.
Then (2)(a) holds by taking $f=\ovl{\al}$ for all $\al\in Q_1(x,y)$.
To show (2)(b), let 
$$
\xymatrix{
C:x=x_0 \ar@{-}[r]^(0.7){l_1}&
x_1 \ar@{-}[r]^{l_2}&
\cdots \ar@{-}[r]^(0.45){l_n}&
x_n=x
}
$$
 be a cycle in $Q$ for some $l_1,\dots,l_n \in \ovl{Q}_1$.
 Let  $i=1,\dots,n$.
 If $\l_{i} \in Q_1(x_{i-1},x_{i})$, then  $E_{l_i}=\ro(x_{i-1})\inv\ro(x_{i})$.
If $l_i=\al_i\inv$ with $\al_{i} \in Q_1(x_{i},x_{i-1})$,
then $E_{\al_i}=\ro(x_i)\inv\ro(x_{i-1})$ and $E_{l_i}=\ro(x_{i-1})\inv\ro(x_{i})$.
In any case $E_{l_i}=\ro(x_{i-1})\inv\ro(x_{i})$.
Hence 
$$E_C=E_{l_1}\cdots E_{l_n}=\ro(x_0)\inv \ro(x_1)\ro(x_1)\inv\ro(x_2)\cdots \ro(x_{n-1})\inv\ro(x_n)=1.$$

(2) \implies (1).
Let $\{\calX_i \mid i \in I\}$ be the set of all connected components of $Q$.
Choose one vertex $x_i$ in $\calX_i$  for each $i\in I$.
Let $i\in I$. Define $\ro(x_i):=1$.
For each vertex $x$ in $\calX_i$ take a walk $w$ from $x_i$ to $x$ in $\calX_i$ as follows:
$$
\xymatrix{
w: x_i=y_0 \ar@{-}[r]^(0.7){l_1}&
y_1 \ar@{-}[r]^{l_2}&
\cdots \ar@{-}[r]^(0.45){l_m}&
y_m=x.
}
$$
 Then we define $\ro(x):=E_{l_1}\cdots E_{l_m}=E_w$.
  This is well-defined because of (2)(b).
  Indeed, let $w'$ be another walk from $x_i$ to $x$.
  Then $C:={w'}\inv w$ is a cycle through $x$, hence we have 
  $1=E_C=E_{w'}\inv E_w$, which shows that $E_w=E_{w'}$.
Let $f \in R(x,y)$ with $x,y \in R_0$.
We have only to show $E(f)=\ro(x)\inv\ro(y)f$.
We may assume that $f=\ovl\mu $ for some path $\mu=\al_n\cdots\al_1$ in $Q$.
Let $\calX_i$ be the connected component of $Q$ containing $x$.

(i) When $f=\ovl\al$ for some $\al\in Q_1$.
Let $w$ be a walk from $x_i$ to $x$.
Then
$\ro(x)=E_w$ and $\ro(y)=E_{\al w}=E_{\al}E_w$. Thus $c_{x,y}=E_{\al}=\ro(x)\inv\ro(y)$.
Hence $E(\ovl\al)=\ro(x)\inv\ro(y)\ovl\al$.

(ii) Otherwise, we have
$$
\begin{aligned}
E(\ovl\mu)&=E(\ovl\al_n)\cdots E(\ovl\al_1)\\
&=E_{\al_n}\ovl\al_n\cdots E_{\al_1}\ovl\al_1\\
&=E_{\al_1}\cdots E_{\al_n}\ovl\mu \\
&=\ro(z_0)\inv \ro(z_1)\ro(z_1)\inv\ro(z_2)\cdots \ro(z_{n-1})\inv\ro(z_n)\ovl\mu \\
&=\ro(x)\inv\ro(y)\ovl\mu,
\end{aligned}
$$
where
$$
\xymatrix{
\mu:x=z_0 \ar[r]^(0.7){\al_1}&
z_1 \ar@{->}[r]^{\al_2}&
\cdots \ar@{->}[r]^(0.45){\al_n}&
z_n=y.
}
$$

(3) \equivalent (4). This follows by Theorem \ref{orbitcat-eq}.
More precisely, (3) implies (4) because 
$$(\blank)/\bbZ\colon\ZCat\to\ZGrCat$$
is a 2-functor, and (4) implies (3) by the following strictly commutative diagram:
$$
\xymatrix@C=60pt{
(R/\ang{g})\#\bbZ & (R/\ang{h})\#\bbZ\\
(R, g\upbl) & (R, h\upbl)
\ar"1,1";"1,2"^{((\id_R,\et)/\bbZ)\#\bbZ}
\ar"2,1";"2,2"_{(\id_R,\et)}
\ar"2,1";"1,1"^{\ep_{(R,g\upbl)}}
\ar"2,2";"1,2"_{\ep_{(R,h\upbl)}}
}
$$
and the fact that  $\ep'_{(R,h\upbl)}\ep_{(R,h\upbl)}=\id_{(R,h\upbl)}$
(see \cite[Theorem 7.5]{Asa} ).

(1) \implies (3).
Assume that the statement (1) holds.
Then by applying (1) to $g\inv(f) \colon g\inv(x) \to g\inv(y)$
we have 
\begin{equation} \label{eq:inv}
\ro(g\inv(y))\inv g\inv(f) = h\inv(f)\ro(g\inv(x))\inv.
\end{equation} 
Now we have to construct a $\bbZ$-equivariant equivalence
$(\id_R,\et)\colon (R,g\upbl) \to (R,h\upbl)$.
The statement (1) yields the following commutative diagram:  
\begin{equation} \label{eq:def-eta}
{\tiny\vcenter{\xymatrix@C=55pt{
g^n(x)&g^{n-1}h(x)&\cdots&gh^{n-1}(x)& h^n(x) \\
g^n(y)&g^{n-1}h(y)&\cdots&gh^{n-1}(y)& h^n(y)
\ar"1,1";"1,2"^{\ro(x)\id}
\ar"2,1";"2,2"_{\ro(y)\id}
\ar"1,1";"2,1"_{g^n(f)}
\ar"1,2";"2,2"^{g^{n-1} h(f)}
\ar"1,2";"1,3"^{\ro(g(x))\id}
\ar"2,2";"2,3"_{\ro(g(y))\id}
\ar"1,3";"1,4"^{\ro(g^{n-2}(x))\id}
\ar"2,3";"2,4"_{\ro(g^{n-2}(y))\id}
\ar"1,4";"2,4"_{gh^{n-1}(f)}
\ar"1,4";"1,5"^{\ro(g^{n-1}(x))\id}
\ar"2,4";"2,5"_{\ro(g^{n-1}(y))\id}
\ar"1,5";"2,5"^{h^n(f)}
}}}
\end{equation}
for all $n > 0$, and

\begin{equation} \label{eq:def-eta2}
{\tiny\vcenter{\xymatrix@C=52pt{
g^n(x)&g^{n+1}h\inv(x)&\cdots&g\inv h^{n+1}(x)& h^n(x) \\
g^n(y)&g^{n+1}h\inv(y)&\cdots&g\inv h^{n+1}(y)& h^n(y)
\ar"1,1";"1,2"^{\ro(g\inv(x))\inv\id}
\ar"2,1";"2,2"_{\ro(g\inv(y))\inv\id}
\ar"1,1";"2,1"_{g^n(f)}
\ar"1,2";"2,2"^{g^{n+1} h\inv(f)}
\ar"1,2";"1,3"^{\ro(g^{-2}(x))\inv\id}
\ar"2,2";"2,3"_{\ro(g^{-2}(y))\inv\id}
\ar"1,3";"1,4"^{\ro(g^{n+1}(x))\inv\id}
\ar"2,3";"2,4"_{\ro(g^{n+1}(y))\inv\id}
\ar"1,4";"2,4"_{g\inv h^{n+1}(f)}
\ar"1,4";"1,5"^{\ro(g^{n}(x))\inv\id}
\ar"2,4";"2,5"_{\ro(g^{n}(y))\inv\id}
\ar"1,5";"2,5"^{h^n(f)}
}}}
\end{equation}
for all $n < 0$.

By looking at \eqref{eq:inv}, \eqref{eq:def-eta} and \eqref{eq:def-eta2}
we define a family $\et =(\et _n)_{n \in \bbZ}$ of natural isomorphisms $\et _n:g^n \To  h^n$ by
\[ \et _{n,x}:=\begin{cases}
            \ro (x)\ro (g(x))\cdots \ro (g^{n-1}(x))\id _{g^n(x)} & \text{if } n>0 \\
            \id _x & \text{if } n=0\\
            \ro (g^n(x))^{-1}\ro (g^{n+1}(x))^{-1}\cdots \ro (g^{-1}(x))^{-1}\id _{g^n(x)} & \text{if } n<0
            \end{cases} \]
for all $n\in \bbZ$ and $x\in R_0$.
To verify that $(\id_R,\et)$ is a $\bbZ$-equivariant equivalence,
it is enough to show
\begin{equation}\label{eq:equivar}
\et _{m+n,x}=\et_{m,h^n(x)}g^m(\et _{n,x})\ (=\et_{m,g^n(x)}g^m(\et _{n,x}))
\end{equation}
for all $m,n\in \bbZ$ and $x\in R_0$. 
We may assume that $n\neq 0$ because \eqref{eq:equivar} is trivial for $n=0$. 
The commutative diagrams \eqref{eq:def-eta} and \eqref{eq:def-eta2} show that
\begin{align}
\label{eq:positive}
\et _{n,x}
&
=\et_{n-1,g(x)}g^{n-1}(\et _{1,x}) &(n > 0),\\
\label{eq:negative}
\et _{n,x}
&
 =\et_{n+1,g^{-1}(x)}g^{n+1}(\et _{-1,x}) &(n < 0).
\end{align}
Then for each $n > 0$ 
we use \eqref{eq:positive} to show \eqref{eq:equivar} as follows:
\begin{eqnarray*}
\et _{m+n,x}&=&\et_{m+n-1,g(x)}g^{m+n-1}(\et _{1,x})\\
                 &=&\et_{m+n-2,g^2(x)}g^{m+n-2}(\et _{1,g(x)})g^{m+n-1}(\et _{1,x})\\
                 &=&\et_{m+n-2,g^2(x)}g^{m+n-2}(\et _{1,g(x)}g(\et _{1,x}))\\
                 &=&\et_{m+n-2,g^2(x)}g^{m+n-2}(\et _{2,x})\\
                 &\vdots & \\
                 &=&\et_{m,g^n(x)}g^m(\et _{n,x}).
\end{eqnarray*}
Similarly for each $n <0$ the equality \eqref{eq:equivar} follows by using
\eqref{eq:negative}.

Hence by Theorem \ref{orbitcat-eq}, 
we have an equivalence $(\id_R, \et)/\bbZ \colon R/\ang{g} \to R/\ang{h}$ as $\bbZ$-graded categories that is the identity on the objects,  
thus $(\id_R, \et)/\bbZ $ is an isomorphism.

(3) \implies (1).
Assume that $R$ has no nonzero oriented cycles.
By (3), $\et\inv:g\To h$ is a natural isomorphism.
Then for each $f:x\to y$ in $R$ we have a commutative diagram

$$
\xymatrix{
g(x) & h(x) \\
g(y) & h(y).
\ar "1,1";"1,2"^{\et_x\inv}
 \ar "2,1";"2,2"_{\et_y\inv}
 \ar "1,1";"2,1"_{g(f)}
 \ar "1,2";"2,2"^{h(f)}
}
$$
By the assumption, for each $z \in R_0$, there exists some $\ro(z) \in \k^\times$ 
such that $\et_z\inv=\ro(z)\id_{g(z)}$.
Hence $\ro(y)g(f)=h(f)\ro(x)$.
\end{proof}
 
\begin{cor}
Let $R$ be a locally bounded category having no nonzero oriented cycles, 
$Q$ the ordinary quiver of $R$,
and $g, h, E$ automorphisms of $R$. 
Then the following are equivalent.

\begin{enumerate}
\item
\begin{enumerate}
\item
$Eg$ and $hE$ coincide on the objects of $R$.
\item
There exists a map $\ro \colon R_0 \to \k^{\times}$ such that
$\ro(y)(Eg)(f) = (hE)(f)\ro(x)$ for all morphisms $f\in R(x,y)$ and for all $x,y \in R_0$.
\end{enumerate}
\item
\begin{enumerate}
\item
$Eg$ and $hE$ coincide on the objects of $R$.	
\item
For each $x, y \in Q_0$ and each $\al \in Q_1(x,y)$,
there exists some $c_{x,y} \in \k^\times$
such that
$((Eg)\inv hE)(\ovl{\al}) = c_{x,y}\ovl{\al}$, and
\item
for each cycle $C = l_n\cdots l_1$ $(l_1, \dots, l_n \in \ovl{Q}_1)$
in $Q$, we have $((Eg)\inv hE)_C = 1$.
\end{enumerate}
\item
There exists  a natural isomorphism $\et  \colon hE \To Eg$ such that $(E, \et)\colon (R,g\upbl) \to (R,h\upbl)$ is a $\bbZ$-equivariant equivalence.
\item
There exists  a natural isomorphism $\et  \colon hE \To Eg$ such that 
$(E, \et)/\bbZ \colon R/\ang{g} \to R/\ang{h}$ is a strictly degree preserving, homogeneously dense equivalence of $\bbZ$-graded categories.
\end{enumerate}
\end{cor}

\begin{proof}
(1) \equivalent (2). This immediately follows from Proposition \ref{orbitcat-iso} if we replace $g, h$ by $Eg, hE$, respectivity.

(1) \equivalent (3). This follows from Proposition \ref{orbitcat-iso} by regarding the diagram
 in the left hand side as that in the right:
$$
\xymatrix{
R & R\\
R & R,
\ar"1,1";"1,2"^{E}
\ar"2,1";"2,2"_{E}
\ar"1,1";"2,1"_{g}
\ar"1,2";"2,2"^{h}
\ar@{=>}"1,2";"2,1"_\et
}\quad
\xymatrix{
R & R\\
R & R
\ar"1,1";"1,2"^{\id_R}
\ar"2,1";"2,2"_{\id_R}
\ar"1,1";"2,1"_{Eg}
\ar"1,2";"2,2"^{hE}
\ar@{=>}"1,2";"2,1"_\et
}
$$
(3) \equivalent (4). This follows from the equivalence of $(3)$ and $(4)$ in Proposition \ref{orbitcat-iso}
by replacing $\id_R$ to $E$.
\end{proof}

\begin{rmk} \label{rmk:simple-cycle}
(1) Proposition \ref{orbitcat-iso} does not assume that the $G$-actions are free.
Therefore we extend the derived equivalence classification given in \cite{AK} to
that of all $n$-fold extensions including the case that $n=0$.

(2) Let $\calS$ be a complete set  of representatives of equivalence classes of simple cycles in $Q$.
To verify the condition (2)(b) in Proposition \ref{orbitcat-iso} 
it is enough to show that $(g\inv h)_C = 1$ for each cycle $C \in \calS$.
\end{rmk}

\section{Main results}
Throughout this section we assume that $A$ is an algebra without nonzero oriented cycles unless otherwise stated.

\begin{thm}\label{gme-iso}
Let $\ph$ and $\ps$ be automorphisms of $\hat{A}$ with jump $n \in \bbZ$ that coincide on the objects of $\hat{A}$.
Any map $\ro _0\colon A_0 \to \k^{\times}$ such that
\begin{equation}\label{eq:level0}
\ro _0(y)\ph _0(a)=\ps _0(a)\ro _0(x)\ (a\in A(x,y),\ x,y\in A_0)
 \end{equation}
 can be extended to a map $\ro:\hat{A}_0\to \k ^\times$ such that
 \begin{equation}\label{eq:whole-hatA}
 \ro (v)\ph (f)=\ps(f)\ro (u)\quad (f\in \hat{A}(u,v),\ u,v\in\hat{A}_0),
 \end{equation}
 where  we regard the functor $\id^{[0]}:A\to A^{[0]}$ as the inclusion $A\incl \hat{A}$.
 
Hence if there exists a map $\ro_0$ with \eqref{eq:level0}, 
then by Proposition \ref{orbitcat-iso}, $(\id_{\hat{A}}, \et )/\bbZ \colon \hat{A}/\ang{\ph} \to \hat{A}/\ang{\ps}$ is an equivalence in $\ZGrCat$ for some natural isomorphism $\et\colon\ps\To\ph$, in particular, it is an isomorphism as $\bbZ$-graded categories when $n\ne 0$.
\end{thm}

\begin{proof}
Let $\ro _0\colon A_0 \to \k^{\times}$ be a map with the property \eqref{eq:level0}.
Then we construct an extension $\ro:\hat{A}_0\to \k ^\times$ of $\ro_0$ 
with the property \eqref{eq:whole-hatA}, namely
\begin{align}
\label{eq:same-level}
\ro (y^{[i]})\ph (a^{[i]})&=\ps(a^{[i]})\ro (x^{[i]})\quad (a\in A(x,y))\\
\label{eq:diff-level}
\ro (y^{[i+1]})\ph (\be^{[i]})&=\ps(\be^{[i]})\ro (x^{[i]})\quad (\be\in D(A(y,x)))
\end{align}
for all $x, y \in A_0$ and $i \in \bbZ$.
It follows from \eqref{eq:level0} that
$\ph _0(a) = \ro _0(y)\inv\ps _0(a)\ro _0(x)
= \ps _0\,\xi(\ro _0)(a)$ for all $a \in A(x,y)$ and $x, y \in A_0$
(see Definition \ref{dfn:xi} for $\xi$), and hence
$\ph _0= \ps _0\,\xi(\ro _0)$.
Here we regard $\ro _0$ as the sequence
$(\ro _0(x))_{x\in A_0}\in (\k ^\times)^{A_0}$.
Then we have
$(\ph ^{ -1}\ps\, \widehat{\xi(\ro _0)}\,)_0
=\ph_0 ^{ -1}\ps_0\, (\widehat{\xi(\ro _0)})_0
= \ph_0 ^{ -1}\ps_0\, \xi(\ro _0)
=\id _A$
(see Lemma \ref{repet-auto} for the definition of
$\widehat{\xi(\ro _0)} \in \Aut(\hat{A})$),
and hence there exists an element $\la = (\la_i)_{i\in \bbZ}\in (\k^\times)^{\hat{A}_0}$ such that
$\ph ^{-1}\ps\, \widehat{\xi(\ro _0)}=\Ph (\la)$,
from which we have
\begin{equation}\label{eq:ph-ps}
\ps = \ph\Ph(\la)(\widehat{\xi(\ro _0)})^{-1}.
\end{equation}
We set $\ph =(\ph_i,\bph_i)_{i\in\bbZ}$,
$\ps =(\ps_i,\bps_i)_{i\in\bbZ}$ and
$\Phi(\la)=(\si_i,\bar{\si}_i)_{i\in\bbZ}$.
Further $\widehat{\xi(\ro_0)}=(\xi(\ro_0),\xi(\ro_0))_{i\in\bbZ}$.
By comparing the first entries of the equality \eqref{eq:ph-ps} we see that
$\ps_i = \ph_i\si_i\xi(\ro _0)^{-1}$ for all $i \in \bbZ$.
For each $ a \in A(x,y)$ ($x,y \in A$)
by using equalities \eqref{eq:def-xi} and \eqref{eq:phi}
we have 
$$
\ps_i(a)=\ph_i(a)\la_0(x)\dots\la_{i-1}(x)
\la_0(y)^{-1}\dots\la_{i-1}(y)^{-1}\ro_0(x)^{-1}\ro_0(y)
$$
and
$$
\ps_i(a)\la_0(x)\inv\dots\la_{i-1}(x)\inv\ro_0(x)
=\ph_i(a)\la_0(y)^{-1}\dots\la_{i-1}(y)^{-1}\ro_0(y)
$$
if $i >0$.
By looking at this
 we define $\ro:\hat{A}_0\to \k ^\times$ by the formula
\[ \ro (x^{[i]}):=\begin{cases}
            (\la _0\cdots \la _{i-1}(x))^{-1}\ro _0(x) & \text{if } i>0\\
            \ro _0(x) & \text{if } i=0\\
            \la _i\cdots \la _{-1}(x)\ro _0(x) & \text{if } i<0
            \end{cases} \]
for all $x^{[i]}\in \hat{A}_0$.
Then $\ro$ is certainly an extension of $\ro_0$.
Now let $a\in A(x,y)\  (x, y \in A_0)$ and $i \in \bbZ$.
Note that we have $\ph(a^{[i]})=\ph_i(a)$ by definition of $\ph_i$.
Then by definition of $\ro$ we have
$\ps(a^{[i]})\ro (x^{[i]})=\ro (y^{[i]})\ph (a^{[i]})$,
thus \eqref{eq:same-level} holds.

We next show the equality \eqref{eq:diff-level}.
By comparing the second entries of the equality \eqref{eq:ph-ps} 
we have $\bps_i = \bph_i\bar{\si}_i\xi(\ro _0)^{-1}$.
By using equalities \eqref{eq:def-xi}, \eqref{eq:phi} and \eqref{eq:phibar}
we see that $\bph_i(a)\ro(x^{[i]})=\ro(y^{[i+1]})\bps_i(a)$ for all $a\in A(y,x)\  (x, y \in A_0)$.
we put 
$$\bph_{i;y,x}:=\bph_{i}|_{A(y,x)}\colon A(y,x)\to A(\ph_i(y),\ph_{i+1}(x))$$
and 
$$\bps_{i;y,x}:=\bps_{i}|_{A(y,x)}\colon A(y,x)\to A(\ps_i(y),\ps_{i+1}(x)).$$
Then we have $\bph_{i;y,x}\ro(x^{[i]})=\ro(y^{[i+1]})\bps_{i;y,x}$ for all $x, y \in A_0$.
Hence 
$$\ro(y^{[i+1]})\bph_{i;y,x}\inv=\bps_{i;y,x}\inv\ro(x^{[i]}),$$
and for each $\be \in D(A(y,x))$ we have
$$\be\ro(y^{[i+1]})\bph_{i;y,x}\inv=\be\bps_{i;y,x}\inv\ro(x^{[i]}),$$
thus 
$$\ro(y^{[i+1]})D(\bph_{i;y,x}\inv)(\be)=D(\bps_{i;y,x}\inv)(\be)\ro(x^{[i]})$$ and
$$\et_i\ro(y^{[i+1]})D(\bph_{i;y,x}\inv)(\be)=\et_i D(\bps_{i;y,x}\inv)(\be)\ro(x^{[i]}).$$
By the commutativity of the diagram \eqref{eq:dfn-phi-2} we have
$$\ro(y^{[i+1]})\ph(\be^{[i]})= \ps(\be^{[i]})\ro(x^{[i]}),$$
as required.
\end{proof}

The assumption of the statement above can be slightly weakened as follows.

\begin{prp}\label{iso}
Let $\ph $ and $\ps $ be automorphisms of $\hat{A}$ with jump $n \in \bbZ$ that coincide on the objects of $\hat{A}$.
We may set $\ph\nu_A^{-n}= (\ph_i,\bph_i)_{i\in\bbZ}$ and $\ps\nu_A^{-n}= (\ps_i,\bps_i)_{i\in\bbZ}$ because they are automorphisms of $\hat{A}$ with jump $0$. 
If there exist $i,j\in \bbZ$ and $\ro :A_0\to \k ^\times$ such that
\begin{equation} \label{eq:comm-ij}
\ro(y)\ph _i(a)=\ps _j(a)\ro(x)
\end{equation}
 for all $a\in A(x,y)$ $x, y \in A_0$, 
 then 
 the equality \eqref{eq:level0} holds.
\end{prp}

\begin{proof}
Since $ \ph\nu_A^{-n}$ and $ \ps\nu_A^{-n}$ are automorphism of $\hat{A}$ with jump 0, 
there exist $(\la,\si), (\mu,\ta)\in (\k^\times)^{\hat{A}_0} \times \Aut (A)$ such that
$$
 \ph\nu_A^{-n} = \hat{\si}\circ \Ph(\la) \text{ and }  \ps\nu_A^{-n} = \hat{\ta}\circ \Ph(\mu).
$$
Then by the equality \eqref{eq:ind-step} in  Remark \ref{jump-rmk} (1) we have  the following commutative diagrams:
$$
\xymatrix@C=45pt{
\ph_i(x)&\ph_{i-1}(x)&\cdots&\ph_1(x)& \ph_0(x) \\
\ph_i(y)&\ph_{i-1}(y)&\cdots&\ph_1(y)&\ph_0(y)
\ar"1,1";"1,2"^{\la_{i-1}(x)\id_x}
\ar"2,1";"2,2"_{\la_{i-1}(y)\id_y}
\ar"1,1";"2,1"_{\ph_{i}(a)}
\ar"1,2";"2,2"^{\ph_{i-1}(a)}
\ar"1,2";"1,3"^{\la_{i-2}(x)\id_x}
\ar"2,2";"2,3"_{\la_{i-2}(y)\id_y}
\ar"1,3";"1,4"^{\la_1(x)\id_x}
\ar"2,3";"2,4"_{\la_1(y)\id_y}
\ar"1,4";"2,4"_{\ph_1(a)}
\ar"1,4";"1,5"^{\la_0(x)\id_x}
\ar"2,4";"2,5"_{\la_0(y)\id_y}
\ar"1,5";"2,5"_{\ph_0(a)}
}
$$
if $i>0$, and
$$
\xymatrix@C=45pt{
\ph_i(x)&\ph_{i+1}(x)&\cdots&\ph_{-1}(x)& \ph_0(x) \\
\ph_i(y)&\ph_{i+1}(y)&\cdots&\ph_{-1}(y)&\ph_0(y) ,
\ar"1,1";"1,2"^{\la_{i}(x)\inv\id_x}
\ar"2,1";"2,2"_{\la_{i}(y)\inv\id_y}
\ar"1,1";"2,1"_{\ph_{i}(a)}
\ar"1,2";"2,2"^{\ph_{i+1}(a)}
\ar"1,2";"1,3"^{\la_{i+1}(x)\inv\id_x}
\ar"2,2";"2,3"_{\la_{i+1}(y)\inv\id_y}
\ar"1,3";"1,4"^{\la_{-2}(x)\inv\id_x}
\ar"2,3";"2,4"_{\la_{-2}(y)\inv\id_y}
\ar"1,4";"2,4"_{\ph_{-1}(a)}
\ar"1,4";"1,5"^{\la_{-1}(x)\inv\id_x}
\ar"2,4";"2,5"_{\la_{-1}(y)\inv\id_y}
\ar"1,5";"2,5"_{\ph_0(a)}
}
$$
if $i<0$.
For each $\ga \in (\k ^\times)^{\hat{A}_0}$ and each $z \in A_0$ we set
$$
\bar{\ga}_i(z) := 
\begin{cases}
(\ga_0 \cdots \ga_{i-1} )(x) & \text{if } i>0,  \\
1 & \text{if } i=0, \text{and}\\
((\ga_i \cdots \ga_{-1} )(x))\inv & \text{if } i<0. 
\end{cases}
$$
When $i,j>0$, it follows from these diagrams that 
$$
\ro (y)\ph _i(a)=\ro (y) (\la _0\cdots \la _{i-2}\la _{i-1}(y))^{-1}\la _0\cdots \la _{i-2}\la _{i-1}(x)\ph_0 (a)
=\ro(y) (\bar{\la}_i (y))\inv \bar{\la}_i (x) \ph_0(a)
$$
and similarly
$$
\ps _j(a)\ro (x)=(\mu _0\cdots \mu _{j-2}\mu _{j-1}(y))^{-1}\mu _0\cdots \mu _{j-2}\mu _{j-1}(x)\ps _0(a)\ro (x)
= (\bar{\mu}_j (y))\inv \bar{\mu}_j (x) \ps_0(a) \ro(x).
$$
By the assumption \eqref{eq:comm-ij} we have
$$
(\bar{\la}_i (y))\inv \bar{\mu}_j (y) \ro(y) \ph_0(a)
=(\bar{\la}_i (x) )\inv \bar{\mu}_j (x)  \ro(x) \ps_0(a) .
$$
Looking at this 
we define 
$\ro _0:A_0\to \k ^\times$ by
$$
\ro _0(z): =(\bar{\la}_i (z) )\inv \bar{\mu}_j (z)  \ro(z)
$$
for all $z\in A_0$.
Then \eqref{eq:level0} holds in this case. 
The remaining cases are verified similarly.
\end{proof}

\begin{cor}\label{gen-twist}
Let $\ph $ be an automorphism of $\hat{A}$ with jump $0\ne n \in \bbZ$.
Then $\hat{A}/\ang{\ph }$ and $T_{\ph _0}^n(A)$ are isomorphic as $\bbZ$-graded categories.
\end{cor}

\begin{proof}
By the definition of $\hat{\ph _0}$,
we have $\ph _0=(\hat{\ph }_0\nu _{A}^n)_0$,
which can be regarded as the equality \eqref{eq:level0} with $\ro_0 (x) = 1$ for all $x \in A_0$.
Hence $\hat{A}/\ang{\ph }$ and $T_{\ph _0}^n(A)$ are isomorphic as $\bbZ$-graded categories by Theorem \ref{gme-iso}.
\end{proof}
\begin{exm}\label{exm:ext}
Let $A$ be an algebra defined by the quiver
$$
\xymatrix{
1 \ar[r]^\al &2 \ar[r]^\be &3. \\
}
$$
Then $\hat{A}$ is given by the quiver
$$
\hat{Q}:
\xymatrix{
\cdots \ar[r] &
1^{[-1]} \ar[r]^{\al^{[-1]}} &2^{[-1]} \ar[r]^{\be^{[-1]}} &3^{[-1]} \ar[r]^{\ga^{[-1]}} &
1^{[0]} \ar[r]^{\al^{[0]}} &2^{[0]} \ar[r]^{\be^{[0]}} &3^{[0]}\ar[r] & \cdots\\
} 
$$
with relations $\mu =0$ for all path $\mu$ of length 3.
Let $\ps \in \Aut^0(\hat{A})$ defined by $\ps(\al):= \ps_\al \al$ for all arrows $\al\in \hat{Q}$,
where  $\ps_\al \in \k ^\times $ is given by the first row of the following diagram: 
 $${\tiny
\xymatrix{
\cdots \ar[r] &
1^{[-1]} \ar[r]^{1} &2^{[-1]} \ar[r]^{1} &3^{[-1]} \ar[r]^{1} &
1^{[0]} \ar[r]^{2} &2^{[0]} \ar[r]^{3} &3^{[0]}\ar[r]^{1}& 1^{[1]} \ar[r]^{1} &2^{[1]} \ar[r]^{1} &3^{[1]} \ar[r] &\cdots\\
\cdots \ar[r] &
1^{[-1]} \ar[r]^{2} &2^{[-1]} \ar[r]^{3} &3^{[-1]} \ar[r]^{1/6} &
1^{[0]} \ar[r]^{2} &2^{[0]} \ar[r]^{3} &3^{[0]}\ar[r]^{1/6}& 1^{[1]} \ar[r]^{2} &2^{[1]} \ar[r]^{3} &3^{[1]} \ar[r] & \cdots\\
\ar ^{1}"1,2";"2,2" 
\ar ^{2}"1,3";"2,3" 
\ar ^{6}"1,4";"2,4" 
\ar ^{1}"1,5";"2,5" 
\ar ^{1}"1,6";"2,6" 
\ar ^{1}"1,7";"2,7" 
\ar ^{1/6}"1,8";"2,8" 
\ar ^{1/3}"1,9";"2,9" 
\ar ^{1} "1,10";"2,10" 
} }
\vspace{-24pt}
$$
Set $\ph:= \ps \nu_A^{n}$ with $ 0\ne n\in\bbZ$.
Then $\hat{\ph}_0 \in \Aut^0 (\hat{A})$ is given by $\hat{\ph}_0 (\al):= (\hat{\ph}_0) _\al \al$ for all arrows $\al\in \hat{Q}$,
where  $ (\hat{\ph}_0) _\al  \in \k ^\times $ is given by the second row of the diagram above. 
By Corollary \ref{gen-twist}, 
$\hat{A}/\ang{\ph }$ and $T_{\ph _0}^n(A)$ are isomorphic as $\bbZ$-graded categories.
The map $\ro$ in Theorem \ref{gme-iso} is illustrated as  the vertical arrows  in the diagram above.
\end{exm}

\begin{thm} \label{thm:cyc-orbit}
Let $\ph$ and $\ps$ be automorphisms of $\hat{A}$ with jump $n \in \bbZ$ that coincide on the objects of $\hat{A}$, and 
let $\calS$ be a complete set  of representatives of equivalence classes of simple cycles in $Q$.
Assume that 
\begin{enumerate}
\item
For each $x, y \in Q_0$ and each $\al \in Q_1(x,y)$,
there exists some $c_{x,y} \in \k^\times$
such that
$(\ph_0\inv \ps_0)(\ovl{\al}) = c_{x,y}\ovl{\al}$, and
\item
for each cycle $C = l_n\cdots l_1$ $(l_1, \dots, l_n \in \ovl{Q}_1)$
in $\calS$, we have 
$(\ph_0\inv \ps_0)_C = 1$ $($see Proposition \ref{orbitcat-iso} for the notation$)$.
\end{enumerate} 
Then $(\id_{\hat{A}}, \et )/\bbZ \colon \hat{A}/\ang{\ph} \to \hat{A}/\ang{\ps}$ is an equivalence in $\ZGrCat$ for some natural isomorphism $\et\colon\ps\To\ph$, in particular, it is an isomorphism as $\bbZ$-graded categories when $n\ne 0$.
\end{thm}

\begin{proof}
This follows easily 
by applying Proposition \ref{orbitcat-iso} and Remark \ref{rmk:simple-cycle}(2) to $A$
and by using Theorem \ref{gme-iso}. 
\end{proof}

\begin{exm}\label{exm:2-cycle}
Let $A$ be an algebra defined by the bound quiver
\[
\vcenter{\xymatrix{ 1 & 2 
              \ar @<1mm>^a"1,1";"1,2"
              \ar @<1mm>^b"1,2";"1,1"
}},ab=0=ba.
\]
Define automorphisms $\ph$ and $\ps$ of $A$ as follows:
On $A_0$ we set $\ph =\id _A=\ps$ and
$\ph (a):=\ph _aa$, $\ph (b):=\ph _bb$,
$\ps (a):=\ps _aa$ and $\ps (b):=\ps _bb$
with $\ph _a, \ph _b, \ps _a, \ps _b \in \k ^\times$.
Fix an arbitrary $n \ne 0$.
Then 
the following are equivalent:
\begin{enumerate}
\item
There exists a natural isomorphism $\et\colon\hat{\ps}\To\hat{\ph}$ such that 
$(\id_{\hat{A}}, \nu_A^n \et )/\bbZ \colon T_{\ph }^n(A)\to T_{\ps}^n(A)$ are isomorphisms as $\bbZ$-graded categories; 
\item
There exists a map $\ro_0 :A_0 \to \k ^\times$ satisfying
$\ph _a\ro_0(2)=\ro_0(1)\ps _a$
and
$\ph _b\ro_0(1)=\ro_0(2)\ps _b$; and
\item
$\ph _a\ph _b=\ps _a\ps _b$.
\end{enumerate}

Indeed, by applying the equivalence of (1) and (2) in  Proposition \ref{orbitcat-iso} 
and Remark \ref{rmk:simple-cycle}(2) to $A$, we see that (2) and (3) are equivalent. 

(3) \implies (1).
 This follows by Theorem \ref{thm:cyc-orbit}.

(1) \implies (2). This follows by applying the implication from (4) to (1) in Proposition \ref{orbitcat-iso} to $\hat{A}$. 
\end{exm}

\section{Piecewise hereditary algebras of tree type}
In this section we will apply the results in the previous section to piecewise hereditary algebras of tree type.
We begin with the following lemma.

\begin{lem} \label{lem:fgld-local}
If $B$ is a local algebra of finite global dimension,
then $B$ is isomorphic to $\k$ as an algebra. 
\end{lem}

\begin{proof}
Let $J$ be the Jacobson radical of $B$, $n$ the global dimension of $B$ and set  $d:=\dim_{\k} B$.
Then we have a projective resolution of $\k = B/J$ of the form
$$
0 \to B^{(b_n)} \to \cdots \to B^{(b_1)} \to B \to \k \to 0,
$$
which gives us short exact sequences
$$
0 \to \Om^{i+1} \k \to B^{(b_i)} \to \Om^{i} \k \to 0
$$
for all $i =1,\cdots n$. 
This shows that $\dim_{\k} \Om^{i+1} \k  \equiv -\dim_{\k} \Om^{i} \k  \pmod d\/$.
Hence 
$$
0 = \dim_{\k} \Om^{n+1} \k  \equiv (-1)^n \dim_{\k} \Om^{1} \k \pmod d\/.
$$
Since $0 \leq \dim_{\k} \Om^{1} \k < d$, 
we have $J= \Om^{1} \k =0$. 
\end{proof}

The following lemma enables us to apply the statement above to piecewise hereditary algebras of tree type.

\begin{lem}\label{pwh-nocyc}
A piecewise hereditary algebra has no nonzero oriented cycles.
\end{lem}

\begin{proof}
If $A$ is a piecewise hereditary algebra, then there is a tilting complex $T$ over a hereditary algebra $H$
such that $A\iso \End(T)$.
For each primitive idempotent $e$ in $A$, $eAe$ is isomorphic to $\End(T_e)$
where $T_e$ is a direct summand of $T$.
By \cite[Corollary 5.5]{HKL},
$eAe$ is a piecewise hereditary algebra
because $T_e$ is a partial tilting complex.
Since piecewise hereditary algebras have finite global dimension and $eAe$ is local,
$eAe$ is isomorphic to $\k$ by Lemma \ref{lem:fgld-local}.
Hence $A$ have no nonzero oriented cycles if $A$ is a piecewise hereditary algebra.
\end{proof}

By Corollary \ref{gen-twist} and Lemma \ref{pwh-nocyc} we finally have the following, which gives an
affirmative answer to the conjecture in the introduction.

\begin{cor}\label{cor:answer-conj}
Let $A$ be a piecewise hereditary algebra
and $\ph $ an automorphism of $\hat{A}$ with jump $0 \ne n \in \bbZ$.
Then $\hat{A}/\ang{\ph }$ and $T_{\ph _0}^n(A)$ are isomorphic as $\bbZ$-graded categories.
\end{cor}


\begin{thebibliography}{99}

\bibitem{Asa99} H.\ Asashiba, {\em The derived equivalence classification of representation-finite
selfinjective algebras}, J. Alg., {\bf 214} (1999), 182--221.

\bibitem{Asa02} H.\ Asashiba,
{\em Derived and stable equivalence classification of twisted multifold extensions of piecewise hereditary algebras of tree type},
J.\ Algebra {\bf 249}, (2002), 345--376.

\bibitem{Asa11} H.\ Asashiba,
{\em A generalization of Gabriel's Galois covering functors and derived equivalences},
J.\ Algebra {\bf 334} (2011), 109--149.

\bibitem{Asa}H.\ Asashiba,
{\em A generalization of Gabriel's Galois covering functors and derived equivalences,
II: $2$-categorical Cohen-Montgomery duality}, Applied Categorical Structures {\bf 25} (2017), no. 2, 155--186, DOI: 10.1007/s10485-015-9416-9.

\bibitem{AK} H.\ Asashiba, and M.\ Kimura,
{\em Derived equivalence classification of generalized multifold extensions of piecewise hereditary algebras of tree type},
 Algebra Discrete Math. {\bf 19} (2015), no. 2, 145--161.

\bibitem{Gab} P.\ Gabriel,
{\em The universal cover of a representation-finite algebra}, Representations of algebras (Puebla, 1980), pp. 68--105,
Lecture Notes in Math., {\bf 903} (Springer, Berlin-New York, 1981).

\bibitem{HKL} L.\ A.\ H\"ugel, S.\ Koenig, and Q.\ Liu,
{\em Jordan-H\"older theorems for derived module categories of piecewise hereditary algebras},
J.\ Algebra {\bf 352} (2012), 361--381.

\bibitem{OTY} Y.\ Ohnuki, K.\ Takeda, and K.\ Yamagata,
{\em Automorphisms of repetitive algebras},
J.\ Algebra {\bf 232} (2000), 708--724.

\bibitem{Sa02} M.\ Saor\'in,
{\em Automorphism groups of trivial extensions},
J.\ Pure Appl.\ Algebra {\bf 166} (2002), no. 3, 285-305.

\bibitem{Sa} M.\ Saor\'in,
{\em Automorphism groups of trivial extension and repetitive algebras}
(unpublished paper).

\end{thebibliography}
\end{document}